\documentclass[letterpaper, 11pt,  reqno]{amsart}

\usepackage{amsmath,amssymb,amscd,amsthm,amsxtra, esint}%,eufrak
\usepackage{tikz} 
\usepackage{color}
\usepackage[implicit=true]{hyperref}

\usepackage{enumitem}

\usepackage{mathtools}

\usepackage{cases}
\usepackage{verbatim} 

\allowdisplaybreaks[2]

\usetikzlibrary{shapes.misc}
\usetikzlibrary{shapes.symbols}
\usetikzlibrary{shapes.geometric}

\tikzset{
	ddot/.style={circle,fill=white,draw=black,inner sep=0pt,minimum size=0.8mm},
	>=stealth,
	}

\tikzset{
	ddot2/.style={circle,fill=black,draw=black,inner sep=0pt,minimum size=0.8mm},
	>=stealth,
	}

\newtheorem{theorem}{Theorem} [section]

\newtheorem{lemma}[theorem]{Lemma}
\newtheorem{proposition}[theorem]{Proposition}
\newtheorem{remark}[theorem]{Remark}

\DeclareMathOperator*{\supp}{supp}

\newcommand{\1}{\hspace{0.5mm}\text{I}\hspace{0.5mm}}
\newcommand{\noi}{\noindent}
\newcommand{\Z}{\mathbb{Z}}
\newcommand{\R}{\mathbb{R}}
\newcommand{\C}{\mathbb{C}}
\newcommand{\T}{\mathbb{T}}

\let\Re=\undefined\DeclareMathOperator*{\Re}{Re}
\let\Im=\undefined\DeclareMathOperator*{\Im}{Im}

\newcommand{\E}{\mathbb{E}}

\def\norm#1{\|#1\|}

\newcommand{\eps}{\varepsilon}

\newcommand{\s}{\sigma}

\newcommand{\dt}{\partial_t}

\newcommand{\Plow}{P_{N}}
\newcommand{\Slow}{S_{N}}

\renewcommand{\l}{\ell}

\newcommand{\N}{\mathbb{N}}

\newtheorem*{ackno}{Acknowledgements}

\numberwithin{equation}{section}
\numberwithin{theorem}{section}

\begin{document}
\baselineskip = 15pt

\title[GWP for the 2-d SCGL Equation]
{Global well-posedness for the two-dimensional stochastic complex Ginzburg-Landau equation}

\author[W. J. Trenberth]
{William J. Trenberth}
\address{
William John Trenberth, School of Mathematics\\
The University of Edinburgh\\
and The Maxwell Institute for the Mathematical Sciences\\
James Clerk Maxwell Building\\
The King's Buildings\\
Peter Guthrie Tait Road\\
Edinburgh\\ 
EH9 3FD\\
 United Kingdom}
\email{}

\begin{abstract}
We study the stochastic complex Ginzburg-Landau equation
(SCGL) with an additive space-time white noise forcing
on the two-dimensional torus. This equation is singular and thus we
need to renormalize the nonlinearity in order to give proper meaning to the
equation. Unlike the real-valued stochastic quantization equation,
SCGL is complex valued and hence we are forced to work with the
generalized Laguerre polynomials for the sake of renormalization. In handling nonlinearities of arbitrary degree,
we derive a useful algebraic identity on the renormalization in the
complex-valued setting and prove that the renormalized SCGL is locally well-posed.
We prove global well-posedness using an energy estimate and almost sure global well-posedness under different conditions using an invariant measure argument.
\end{abstract}

%\date{\today}
\maketitle
\baselineskip = 14pt
\tableofcontents

\section{Introduction}

\subsection{The stochastic complex Ginzburg-Landau equation}

We study
the following stochastic complex Ginzburg-Landau equation on $\T^2 = (\R/2\pi\Z)^2$
with an additive space-time white noise forcing:
\begin{align}
\begin{cases}
\dt u=(a_1+ia_2)\Delta u+(b_1+ib_2)u -(c_1+ic_2)|u|^{2m-2}u+\sqrt{2\gamma}\xi   \\
u|_{t = 0} = u_0 
\end{cases}
\label{SCGL}
\end{align}
\noi
where $(x, t) \in \T^2\times \R_+$, $a_1,\gamma,>0$, $a_2, b_1, b_2, c_2, c_3\in\R$, $m\geq 2$ is an integer and 
$\xi(x, t)$ denotes a complex valued, Gaussian, space-time white noise on $\T^2\times \R_+$. 

For $s > -\frac{2}{2m-1}$, we consider SCGL with initial data in the space $C^s(\T^2)$. Here $C^{s}(\T^2)$ is the Besov-H\"older space of regularity $s$. See Section \ref{Sec: Basic ests} for the definition and some basic properties of these spaces. 

The complex Ginzburg-Landau equation (CGL), namely equation \eqref{SCGL} without the white noise forcing, is one of the most studied equations in physics. CGL is used to describe phenomenon such as nonlinear waves, second order phase transitions, superconductivity, Bose-Einstein condensation and liquid crystals. For more information on complex Ginzburg-Landau equations in physics, see the survey paper \cite{WORLDofCGL}.

Due to its physical importance, CGL has been heavily studied from a mathematical perspective. See for example \cite{DGL} where an energy estimate was used to show that, with twice differentiable initial data, if the ratio $|\frac{a_2}{a_1}|$ is small enough, CGL on $\T^d$ is globally well-posed. See \cite{GV} for a similar result on $\R^d$. 

There has also been a substantial amount of research on complex Ginzburg-Landau equations with random forcing. See for example \cite{Kuksin1997,Hai03,kuksinSher,Kuksin2013,KuksinNer2013,Shir,HIN,Hos,Wei}. We note in particular the work \cite{Hai03} where Hairer studied a complex Ginzburg-Landau equation driven by a real valued space-time white noise in one spatial dimension and the work \cite{HIN} where Hoshino, Inahama and Naganuma studied a complex Ginzburg-Landau equation driven by a complex valued space-time white noise in three spatial dimensions. To the authors knowledge, no work has studied complex Ginzburg-Landau equations driven by a space-time white noise in two spatial dimensions. The current work aims to fill this gap in the literature.

In this paper we will only consider SCGL~\eqref{SCGL}, with $b_1=-a_1$ and $b_2=-a_2$. That is we consider SCGL in the form:
\begin{align}
\dt u=(a_1+ia_2)[\Delta -1]u-(c_1+ic_2)|u|^{2m-2}u+\sqrt{2\gamma}\xi.
\label{SCGLb}
\end{align}
We do this to avoid issues occurring at the zero frequency that arise as $\Delta$ is not a strictly positive operator.  

\subsection{Renormalized SCGL}\label{SUBSEC: Renormalized SCGL}

To explain the renormalization procedure used for this equation, we consider a truncated version of it. Let $B_R$ be the ball of radius $R$ centered at $0$, measured in the Euclidean distance. We denote $\chi_N = \chi(\frac{\cdot}{N})$, where $\chi:\R^2\rightarrow\R$ is a smooth function such that $\chi = 1$ on $B_{\frac{1}{2}}$ and $\chi = 0$ outside of $B_{1}$. Using this smooth truncation,  we define the smooth frequency projector, $\Slow$ by
\begin{equation*}
    \Slow f=\sum\limits_{|n|\leq N}\chi_N(n)\widehat{f}(n)e^{in\cdot x}.
\end{equation*}
Consider the following truncated version of SCGL~\eqref{SCGLb}:
\begin{align}\label{EQU: Trunc SCGL}
\begin{cases}
\dt u_N=(a_1+ia_2)[\Delta-1]u_N-(c_1+ic_2) |u_N|^{2m-2}u_N + \sqrt{2\gamma}\Slow\xi   \\
u_N|_{t = 0} = \Slow u_0.
\end{cases}
\end{align}
It can be shown that for each fixed $N$, this truncated equation is globally well-posed and its solutions are smooth in space-time. We define the truncated stochastic convolution
\begin{align}
    \Psi_N(t)&\stackrel{\text{def}}{=}\sqrt{2\gamma}\int_{-\infty}^t S(t-t')d(\Slow W(t'))\label{conv}\\
    &=\sqrt{2\gamma}\sum\limits_{|n|\leq N}\chi_N(n)e^{i n\cdot x}\int_{-\infty}^te^{-(t-t')(a_1+ia_2)(|n|^2+1)}d\beta_n(t')\nonumber.
\end{align}
Then, using the fact that $\beta_n$ and $\beta_m$ are independent unless $n=m$, and using It\^o's isometry we have,
\begin{align*}
    \sigma_N&=\E\left[|\Psi_N(x,t)|^2 \right]\\
    &=2\gamma\sum\limits_{|n|\leq N}|\chi_N(n)|^2 \int_{-\infty}^te^{-2(t-t')a_1(|n|^2+1)}dt'\\
    &=\sum\limits_{|n|\leq N} |\chi_N(n)|^2 \frac{\gamma}{a_1(|n|^2+1)}\\
    &\sim_{a_1,\gamma} \log N.
\end{align*}
In particular $\sigma_N$ is independent of $(x,t)\in \T^2\times \R_{+}$. It follows that $\Psi_N$ is a Gaussian random variable of mean zero and variance $\sigma_N$. We make the ansatz $u_N=v_N+\Psi_N$ and then study the resulting equation for $v_N$:
\begin{align}\label{EQU: for Trunc v_N SCGL}
\begin{cases}
\dt v_N=(a_1+ia_2)[\Delta -1]v_N-(c_1+ic_2)\hspace{-10pt} \sum\limits_{\substack{0\leq i\leq m\\0\leq j\leq m-1}}\hspace{-10pt}\binom{m}{i}\binom{m-1}{j}v_N^i\overline{v_N}^j\Psi_N^{m-i}\overline{\Psi_N}^{m-j-1}\vspace{-10pt} \\
u_N|_{t = 0} = \Slow u_0-\Psi_ N(0).
\end{cases}
\end{align}
The nonlinearity in the above equation comes from expanding
\begin{equation*}\label{EQU: double binomial}
|x+y|^{2m-2}(x+y)=(x+y)^m \overline{(x+y)^{m-1}}
\end{equation*}
using the binomial theorem twice. This ansatz is one of the main ideas in \cite{DPD} and has come to be known as the \emph{Da Prato-Debussche trick} in the SPDE literature. However, this idea was first used by McKean \cite{McKean} and Bourgain \cite{Bourgain} in the context of PDEs with random initial data.

The equation \eqref{EQU: for Trunc v_N SCGL} still has the problem that the monomials $\Psi_N^k\overline{\Psi_N}^\ell$ do not have good limiting behaviour as $N\rightarrow \infty$. Beleaguered by this lack of convergence, we consider instead the \emph{Wick ordered} truncated monomials defined by\footnote{For the purposes of this paper it is enough to take this as a definition. See \cite{OT1, Nelson} for information on how this relates to Fock spaces.}
\begin{equation}\label{Lag2}
:\!\Psi_N^{k}\overline{\Psi_N}^\ell\!: =\begin{cases}
(-1)^{k}k!L_{k}^{(\ell-k)}(|\Psi_N(x, t)|^2;\sigma_N)\overline{\Psi_N}^{\ell},& k>\ell,\\[10pt]
(-1)^{\ell}\ell!L_{\ell}^{(k-\ell)}(|\Psi_N(x, t)|^2;\sigma_N)\Psi_N^{k},& \ell\leq k.
\end{cases}
\end{equation}
Here $L^{(\ell)}_k(x;\sigma)=\sigma^k L^{(\ell)}_k(\frac{x}{\sigma})$ where $L^{(\ell)}_k(x)$ are the  generalized Laguerre polynomials. These polynomials can be defined using the recursion relation
\begin{equation*}\label{lagrecrel}
L_{k+1}^{(\ell)}(x)=\frac{(2k+1+\ell-x)L_{k}^{(\ell)}(x)-(k+\ell)L^{(\ell)}_{k-1}(x)}{k+1}
\end{equation*}
after initializing
\begin{equation*}
L_0^{(\ell)}(x)=1,\quad L_1^{(\ell)}(x)=1+\ell-x.
\end{equation*}
Alternatively, the generalized Laguerre polynomials can be defined using a generating function:
\begin{equation*}
G_L(t,x,\ell)\frac{1}{(1-t)^{\ell+1}}e^{-\frac{tx}{1-t}}=\sum\limits_{n=0}^\infty t^nL^{(\ell)}_n(x).
\end{equation*}
The first few generalized Laguerre polynomials are:
\begin{align*}
L_0^{(\ell)}(x)&=1\\
L_1^{(\ell)}(x)&=1+\ell-x\\
L_1^{(\ell)}(x)&=\frac{1}{2}x^2-(\ell+2)x+\frac{(\ell+1)(\ell+2)}{2}\\
L_3^{(\ell)}(x)&=\frac{-1}{6}x^3+\frac{(\ell+3)}{2}x^2-\frac{(\ell+2)(\ell+3)}{2}x+\frac{(\ell+1)(\ell+2)(\ell+3)}{6}.
\end{align*}
One can show that, for given $k,\ell\in\mathbb{N}$, the Wick ordered truncated monomial, $:\!\Psi_N^{k}\overline{\Psi_N}^\ell\!:$, converges to a well defined distribution which we denote by $:\!\Psi^k\overline{\Psi}^\ell\!:$. In particular we have the following proposition.

\begin{proposition}\label{PROP:sconv}
Let  $k,\l \in \N$, $T >0$ and $p\ge 1$. 
Then, $\{ :\!\!\Psi^\l_N\overline{\Psi_N}^k \!\!:  \}_{N \in \mathbb{N}}$
is a Cauchy sequence in $L^p(\Omega; C([0,T];C^{-\eps}(\T^2)))$.
Moreover, denoting the limit by  $:\!\Psi^\l\overline{\Psi}^k\!:$,    
we have  $:\!\Psi^\l\overline{\Psi}^k\!:\, \in C([0,T];C^{-\eps}(\T^2))$ almost surely. 
\end{proposition}

We will prove this proposition in Section \ref{SEC: Sto Conv}. One can think of $:\!\Psi^k\overline{\Psi}^\ell\!:$ as being $\Psi^k\overline{\Psi}^\ell$ with infinite counter terms. For example, for $k=2$ and $\ell=1$ we can think of $:|\Psi|^2\Psi:$ as being $|\Psi|^2\Psi-2\infty\Psi$. This heuristic is justified by looking at \eqref{Lag2} and noting that $\sigma_N\rightarrow \infty$ as $N\rightarrow \infty$. 

Consider now the following equation:
\begin{align}\label{EQU: for v SCGL}
\begin{cases}
\dt v=(a_1+ia_2)[\Delta-1]v-(c_1+ic_2)\hspace{-10pt} \sum\limits_{\substack{0\leq i\leq m\\0\leq j\leq m-1}}\hspace{-10pt}\binom{m}{i}\binom{m-1}{j}v^i\overline{v}^j:\!\Psi^{m-i}\overline{\Psi}^{m-j-1}\!:  \\[5pt]
v|_{t = 0} = u_0-\Psi(0).
\end{cases}
\end{align}
This is an untruncated version of \eqref{EQU: for Trunc v_N SCGL} with $\Psi^{m-i}\overline{\Psi}^{m-j-1}$ replaced by $:\!\Psi^{m-i}\overline{\Psi}^{m-j-1}\!:$. The point here is that although $\Psi^{m-i}\overline{\Psi}^{m-j-1}$ is not well defined, $:\!\Psi^{m-i}\overline{\Psi}^{m-j-1}\!:$ is by Proposition \ref{PROP:sconv}.

If a solution $v$ exists to this equation we \emph{define} 
\begin{align*}
    u\stackrel{\text{def}}{=} v+\Psi
\end{align*}
to be a solution of the \emph{Wick ordered} SCGL (WSCGL), 
\begin{align}\label{EQU: renormalised equ for u}
\begin{cases}
\dt u=(a_1+ia_2)[\Delta-1] u -(c_1+ic_2):\!|u|^{2m-2}u\!:+\sqrt{2\gamma}\xi  \\
u_N|_{t = 0} = u_0
\end{cases}
\end{align}
where $:\!|u|^{2m-2}u\!: = (-1)^{m-1}(m-1)!L^{(1)}_{m-1}(|u|^2;\infty)$ is the \emph{Wick ordered nonlinearity}. This perhaps seems unusual because we are defining solutions to WSCGL~\eqref{EQU: renormalised equ for u} through another equation, \eqref{EQU: for v SCGL}. We do this because WSCGL only makes sense formally. WSCGL~\eqref{EQU: renormalised equ for u} is an abuse of notation because  $L^{(1)}_{m-1}(x;\infty)$ makes no sense and is an abuse of definitions because, Wick ordering is only defined for Gaussian random variables, see \cite{Nelson} for more information, and there is no reason for $u$ to be a Gaussian random variable. As the nonlinearity in WSCGL~\eqref{EQU: renormalised equ for u} does not have any rigorous meaning, WSCGL~\eqref{EQU: renormalised equ for u} itself also does not have any rigorous meaning. However, morally, WSCGL~\eqref{EQU: renormalised equ for u} is the renormalized equation that we are trying to solve in this paper. 

This definition of $u$ solving WSCGL~\eqref{EQU: renormalised equ for u} is of such importance to this paper that it is worth stating in a reverse manner for extra clarity: $u$ solves WSCGL~\eqref{EQU: renormalised equ for u} if $v = u-\Psi$ solves \eqref{EQU: for v SCGL}.

The connection between \eqref{EQU: for v SCGL} and WSCGL~\eqref{EQU: renormalised equ for u} can be understood by looking at truncated versions of these equations.
Looking at the following truncated version of \eqref{EQU: for v SCGL}:
\begin{align}\label{EQU: trunc renorm v SCGL}
\begin{cases}
\dt v_N=(a_1+ia_2)[\Delta-1] v_N- (c_1+ic_2)\hspace{-10pt}\sum\limits_{\substack{0\leq i\leq m\\0\leq j\leq m-1}}\hspace{-10pt}\binom{m}{i}\binom{m-1}{j}v_N^i\overline{v_N^j}:\!\Psi^{m-i}_N\overline{\Psi_N}^{m-j-1}\!:  \\[5pt]
v|_{t = 0} = \Slow u_0-\Psi_N(0)
\end{cases}
\end{align}
and using the following generalized Laguerre polynomial sum formula, see Section \ref{SEC: Lag forms},
\begin{align}\label{EQU: Lag sum form}
(-1)^{m}m!L_{m}^{(1)}(|x+y|^2;\sigma)(x+y)=\sum\limits_{\substack{0\leq i\leq m+1\\0\leq j\leq m}} \binom{m+1}{i}\binom{m}{j}P^{m,\sigma}_{i,j}(y,\overline{y})x^i\overline{x}^j
\end{align}
where
\begin{equation*}
P_{i,j}^{m,\sigma}(y,\overline{y})=\begin{cases}
(-1)^{m-j}(m-j)!L^{(j-i+1)}_{m-j}(|y|^2;\sigma)y^{j-i+1},& j+1\geq i\\
\\
(-1)^{m-i+1}(m-i+1)!L^{(i-j-l)}_{m-i+1}(|y|^2;\sigma)\overline{y}^{i-j-l},&j+1\leq i.
\end{cases}
\end{equation*}
it follows that $u_N=v_N+\Psi_N$ satisfies the following equation:
\begin{align}\label{EQU: renormalised equ for u_N}
\begin{cases}
\dt u_N=(a_1+ia_2)[\Delta-1] u_N -(c_1+ic_2)(-1)^{m-1}(m-1)!L^{(1)}_{m-1}(|u_N|^2;\sigma_N) u_N+\sqrt{2\gamma}\Slow\xi  \\
u_N|_{t = 0} = \Slow u_0.
\end{cases}
\end{align}
The Laguerre polynomial sum formula \eqref{EQU: Lag sum form} intermediates \eqref{EQU: renormalised equ for u_N} and \eqref{EQU: trunc renorm v SCGL}. Formally taking a limit as $N\rightarrow \infty$, the relationship, $u_N=v_N+\Psi_N$, between the truncated equations \eqref{EQU: renormalised equ for u_N} and \eqref{EQU: trunc renorm v SCGL} gives justification for defining solutions to the purely formal \eqref{EQU: renormalised equ for u} through \eqref{EQU: for v SCGL}. To the authors knowledge the sum formula \eqref{EQU: Lag sum form} has not appeared in the literature. An elementary proof is given in Section \ref{SEC: Lag forms} of this paper.

\begin{remark}
\textnormal{An alternative way to define solutions to WSCGL~\eqref{EQU: renormalised equ for u} would be the following: $u$ solves WSCGL~\eqref{EQU: renormalised equ for u} if the sequence $\{u_N\}_{N\in\mathbb{N}}$, where $u_N$ solves \eqref{EQU: renormalised equ for u_N}, converges in probability to $u$, measured in the space $C([0,T], C^{s}(\T^2))$.  It turns out that these definitions are equivalent. This is the case for practically all singular SPDEs, see for example \cite[Remark 1.2]{GKO}}.
\end{remark}

\subsection{Main results}
Our main goal in this paper is to study the well-posedness of WSCGL~\eqref{EQU: renormalised equ for u} defined in the previous section. We will prove the following.

\begin{theorem}\label{LWP}
Let $a_1>0$, $m\geq 2$ be an integer, let $s_0>-\frac{2}{2m-1}$ and $\eps>0$ be sufficiently small. Then WSCGL~\eqref{EQU: renormalised equ for u} is pathwise locally well-posed in $C^{s_0}(\T^2)$. More precisely there exists $\theta>0$ such that given any $u_0\in C^{s_0}(\T^2)$, there exists $T\sim_\omega \norm{u_0}^{-\theta}_{C^{s_0}(\T^2)}$, which is positive almost surely, such that there is a unique solution to the mild formulation of \eqref{EQU: for v SCGL} on $[0,T]$ with 
\begin{equation*}
v\in C((0,T];C^{2\eps}(\T^2))\cap C([0,T];C^{s_0}(\T^2)) .
\end{equation*}
\end{theorem}

Using an energy estimate we are then able upgrade this local well-posedness result to deterministic\footnote{Here, and in the rest of this paper, by deterministic global well-posedness we mean deterministic with respect to the initial data. That is, global well-posedness for all initial data as opposed to global well-posedness for almost all initial data.} global well-posedness provided that the dispersion, $a_2$, is small compared to the dissipation, $a_1$, and that the heat part of the nonlinearity is defocusing, $c_1>0$.

\begin{theorem}\label{GWP}
Let $a_1,c_1>0$ and $s_0>-\frac{2}{2m-1}$. Set $r=\big|\frac{a_1}{a_2}\big|$ and let $m\geq 2$ be an integer such that
\begin{equation*}
    2m-1< 2+2\left(r^2+2r\sqrt{1+r^2}\right).
\end{equation*}
Then WSCGL, \eqref{EQU: renormalised equ for u} is pathwise globally well-posed in $C^{s_0}(\T^2)$. More precisely for any $T>0$, and any $u_0\in C^{s_0}(\T^2)$, there almost surely exists a unique solution to the mild formulation of \eqref{EQU: for v SCGL} on $[0,T]$ with 
\begin{equation*}
v\in C((0,T];C^{2\eps}(\T^2))\cap C([0,T];C^{s_0}(\T^2))\quad
\end{equation*}
almost surely, for $\eps>0$ small enough.
\end{theorem}

The above result leaves global well-posedness open for small dissipation-dispersion ratios, $r=|\frac{a_1}{a_2}|$. We extend it to small values of $r$ using an invariant measure argument. 

Intuitively the measure
\begin{equation}\label{EQU: inv meas}
    dP_{a_1,c_1,\gamma,2m}=e^{-\frac{c_1}{2m\gamma}\int :|u|^{2m}:\,dx-\frac{a_1}{2\gamma}\int |\nabla u|^2\,dx-\frac{a_1}{2\gamma}\int|u|^2\,dx}\,du
\end{equation}
should be invariant under the flow of \eqref{EQU: renormalised equ for u} if $\frac{a_2}{a_1}=\frac{c_2}{c_1}$. This is because under this condition the measure \eqref{EQU: inv meas} is expected to be an invariant measure for the heat equation part of \eqref{EQU: renormalised equ for u} and the Schr\"odinger part separately. For a rigorous definition of what is meant by ``invariant measure'' see Section 7. However, as it is written, this measure does not make any sense. In particular the $du$ in \eqref{EQU: inv meas} refers to the non existent Lebesgue measure on some infinite dimensional vector space. Well known results in constructive quantum field theory show that one can rigorously make sense of this measure as a weighted Gaussian measure, see for example \cite{Nelson, OT1}. We review the needed results in Section 7.

Using an invariant measure argument we are able to, in an almost all sense, upgrade Theorem \ref{GWP} to a wider parameter range.

\begin{theorem}\label{Th: AS GWP}
Suppose $\frac{a_2}{a_1}=\frac{c_2}{c_1}$. Then there exists a set, $U\subset \supp P_{a_1,c_1,\gamma,2m}$, of $P_{a_1,c_1,\gamma,2m}$-measure 1 such that, for each $T>0$ and $u_0\in U$, there exists a solution $v$ to \eqref{EQU: for v SCGL} on $[0,T]$, with initial data $u_0$.
\end{theorem}

As a consequence of the arguments used to prove Theorem \ref{Th: AS GWP}, we can also prove the following.

\begin{theorem}\label{THM: invariance}
The measure $P_{a_1,c_1,\gamma,2m}$ is an invariant measure for \eqref{EQU: renormalised equ for u}.
\end{theorem}

\subsection{Methodology and discussion}

In this subsection we describe how the results in this paper, and the methods used to prove them, fit into the wider singular SPDE literature.

Singular stochastic partial differential equations have been heavily studied. The interest in them goes back to an article by Parisi and Wu, \cite{ParWu}, (see also \cite{Nelson}) where it was proposed that the Euclidean $\Phi^k_d$ quantum field theory on finite volume could be constructed as the invariant measure of the aptly named \emph{stochastic quantization equation} (SQE)
\begin{align}\label{SQE}
\begin{cases}
\dt u=\Delta u -u+u^k+\xi   \\
u|_{t = 0} = u_0 
\end{cases}
\quad (x, t) \in \T^d\times \R_+,
\end{align}
where $\xi$ denotes a real valued space-time white noise. More precisely, the Euclidean quantum field theory on finite volume is the invariant measure of an appropriately renormalized version of \eqref{SQE}. SQE~\eqref{SQE} is also referred to as the $\Phi^k_d$-model in the literature. Due to its importance in physics, this equation has been the primary motivator for most major advancements in the field of singular SPDEs.

Concerning the two-dimensional setting, in \cite{DPD} Da Prato and Debussche proved local well-posedness for the Wick ordered SQE using the trick that now bears their names and a fixed point argument. 
In the same paper, almost sure global well-posedness for SQE was proven using an invariant measure argument. In \cite{MW1} an energy estimate was used to prove deterministic global well-posedness, improving the almost sure global well-posedness result in \cite{DPD}. The energy estimate in \cite{MW1} essentially amounts to multiplying the equation \eqref{SQE} by $u^{p-1}$ and integrating in space to get a bound on the growth of the $L^p$ norm of $u$. See \cite{TW,RZZ} for more papers on the two-dimensional SQE involving similar energy estimates.

The situation for the three-dimensional SQE proved to be much more difficult. The Da Prato-Debussche trick and Wick ordering are not sufficient to make sense of SQE in the three-dimensional setting. In the groundbreaking work of \cite{Hai}, Hairer invented the theory of regularity structures, a general framework for making sense of singular stochastic PDEs. Using this machinery, Hairer was able to make sense of and prove the local well-posedness of SQE in three-dimensions. In \cite{MW2} Mouratt and Weber used PDE techniques to show that the three-dimensional SQE `comes down from infinity'. This is a strong result which implies global well-posedness. Since the work of Hairer, there has been a menagerie of results giving alternative frameworks in which to study singular stochastic PDEs. In \cite{CC}  Catellier and Chouk used the theory of paracontrolled distributions developed by Gubinelli, Imkeller and Perkowski in \cite{GIP} and gave an alternative proof of local well-posedness for the three-dimensional SQE. In \cite{Kup}, Kupiainen developed a renormalization group approach to solving SQE.

Many other singular SPDEs fit nicely into the frameworks listed above. For example the KPZ equation 
\begin{equation*}
    \partial_th=\partial_{xx}h+(\partial_x)^2+\xi
\end{equation*}
can be solved in the context of regularity structures, see \cite{Hai} and paracontrolled distributions, see \cite{GubPer}. The two and three-dimensional parabolic Anderson model
\begin{equation*}
    \partial_tu=\Delta u+u\xi
\end{equation*}
can be solved in the context of regularity structures, see \cite{Hai}, or paracontrolled distributions, see \cite{GIP}. In \cite{HIN}, the three-dimensional SCGL was solved in the context of regularity structures and paracontrolled distributions.

As previously mentioned, the aim of this work is to fill in a gap in the literature by solving the two-dimensional SCGL. This problem of course is much  easier than the three-dimensional equation solved in \cite{HIN}. We do not need to use the frameworks of regularity structures or paracontrolled distributions to give meaning to \eqref{SCGLb}. The Da Prato-Debussche trick and Wick ordering suffices, as in \cite{DPD}.

However, at the same time, we believe this gap in the literature is an interesting problem as the renormalization procedure needed is slightly different to that in other papers. Consider for example, SQE. Due to physical considerations one often wants to only study real valued solutions of SQE. To do this one puts the `reality' condition
\begin{equation*}
\widehat{\xi}(n)=\overline{\widehat{\xi}(-n)},\quad \textnormal{ for all } n\in \Z^d
\end{equation*}
on the white noise $\xi$. In contrast to SQE, SCGL has to be studied in the complex valued setting. This is due to the complexifying nature of the Schr\"odinger, $ia_2\Delta u -ic_2|u|^{2m-2}u$, part of the equation. The renormalization procedure for the real valued two-dimensional SQE then differs to the renormalization procedure for the complex valued two-dimensional SCGL described in subsection \ref{SUBSEC: Renormalized SCGL} for two reasons. Firstly, more terms appear in the expansion of $|v+\Psi|^{2m-2}(v+\Psi)$ than in the expansion of $(v+\Psi)^{2m-1}$. The former is a polynomial in two variables $\Psi$ and $\overline{\Psi}$ with coefficients depending on $v$ and $\overline{v}$ while the latter is a polynomial in a single variable, $\Psi$ with coefficients depending on $v$. This means that more terms have to be renormalized in the complex valued setting, compared to the real valued setting. 
Secondly, we use the Laguerre polynomials. This is in contrast to the real valued setting where the Hermite polynomials are used. 

Recall that the generalized Hermite polynomials $H_k(x;\sigma)$ can be defined through the generating function:
\begin{equation*}\label{Hermite Generating Function}
G_H(t,x;\ell)=e^{tx-\frac{1}{2}\ell t^2}=\sum\limits_{k=0}^\infty \frac{t^k}{k!}H_k(x;\ell).
\end{equation*}
When $\ell=1$ we simply call $H_k(x;1)$ the Hermite polynomials and write $H_k(x;1)=H_k(x)$.  The first few generalized Hermite polynomials are 
\begin{align*}
\begin{split}
& H_0(x; \s) = 1, 
\qquad 
H_1(x; \s) = x, 
\qquad
H_2(x; \s) = x^2 - \s,   \\
& H_3(x; \s) = x^3 - 3\s x, 
\qquad 
H_4(x; \s) = x^4 - 6\s x^2 +3\s^2.
\end{split}
\label{H1a}
\end{align*}

For the real valued cubic SQE one essentially replaces the nonlinearity $u^3$ by 
\begin{equation*}
H_3(u;\infty)=u^3-3\infty u.
\end{equation*}
Compare this to the complex valued cubic SCGL were we replace $|u|^2u$ by
\begin{equation*}
-L_1^{(1)}(|u|^2;\infty)=|u|^2u-2\infty u.
\end{equation*}
Note the difference in the amount of `counter terms' subtracted in the above two equations. Formally, this is because in the real valued setting there are three pairings between $u$ and itself in the nonlinearity $u^3$ while in the complex valued setting, there are only 2 pairings between $u$ and $\overline{u}$ in the nonlinearity $|u|^2u = u\overline{u}u$.

Wick ordering has previously been studied in the complex valued setting, see for example \cite{OT1} where the, complex valued, Wick ordered nonlinear Schrodinger equation was studied in the context of random initial data. However, to the authors knowledge, this is the first time Wick ordering and Laguerre polynomials have been applied together in the complex valued setting in the context of singular SPDEs.

In proving the regularity of the Wick ordered monomials, Proposition \ref{PROP:sconv}, we use a Fourier analytic proof, inspired by the analysis in \cite[Proposition 2.1]{GKO}. One key formula used to prove the regularity of the Wick ordered monomials in \cite{DPD,MW1,TW,GKO} and other papers studying real valued equations is the following Hermite polynomial expectation formula:
\begin{equation}\label{hermiteortho intro}
\E[H_k(f;\sigma_f)H_\ell(g;\sigma_g)]=k!\delta_{k\ell}
\end{equation}
where $f$ and $g$ are mean-zero complex valued Gaussian random variables with variances $\sigma_f$ and $\sigma_g$ respectively. In this paper, as we work in the complex valued setting and use Laguerre polynomials, we need the following Laguerre polynomial analogue of \eqref{hermiteortho intro}:
\begin{equation}\label{Lag expectation formula}
\E\left[L^{(\ell)}_k(|f|^2;\sigma_f)f^\ell\overline{L^{(\ell)}_m(|g|^2;\sigma_g)g^\ell} \right]=\delta_{km}\frac{(k+\ell)!}{k!}\left| \E[f\overline{g}]  \right|^{2k}\E[f\overline{g}]^\ell.
\end{equation}
See \cite{Folland} for more information on \eqref{Lag expectation formula}.

The argument used to prove local well-posedness, Proposition \ref{LWP}, in this paper is virtually the same as that in \cite{DPD}: Once we know the Wick ordered monomials have regularity $C^{-\eps}(\T^2)$, we can close a contraction mapping argument to solve \eqref{EQU: for v SCGL} by postulating that $v\in C^{2\eps}(\T^2)$ and using the following product estimate for Besov-H\"older spaces:
\begin{equation*}
    \norm{fg}_{C^{-\eps}(\T^2)}\leq \norm{f}_{C^{-\eps}(\T^2)}\norm{g}_{C^{2\eps}(\T^2)}.
\end{equation*}
See Section \ref{Sec: Basic ests} for a more general statement of this formula. 

The argument used to prove global well-posedness in this paper is similar to the papers \cite{MW1,TW}. From local well-posedness theory it can be shown that the time of local existence, $T$, satisfies $T\gtrsim_\omega \norm{v}_{L^p(\T^2)}^{-\theta}$ for some $\theta>0$ if 
\begin{equation}\label{EQU: GWP cond 1}
2m-1<p.
\end{equation}
It then suffices to get an a priori bound on the growth of $\norm{v}_{L^p}$. To do this we essentially test \eqref{EQU: for v SCGL} by $|v|^{p-2}v$. This argument relies on the nonlinearity of $\eqref{EQU: renormalised equ for u}$ having a good sign, that is $c_1>0$.  However, compared to the real valued setting in \cite{MW1, TW} a complication arises due to the dispersion, the $ia_2\Delta$ term, in \eqref{EQU: renormalised equ for u}. Instead of getting a nice a priori bound of the form
\begin{equation}\label{EQU: a prori bound}
    \partial_t \norm{v}_{L^p(\T^2)}^p\leq C
\end{equation}
as is obtained in \cite{MW1,TW}, we get a bound of the form
\begin{equation*}
    \partial_t \norm{v}_{L^p(\T^2)}^p + 4 A(-2\Im(\overline{v}\nabla v), \nabla |v|^2)\leq C
\end{equation*}
where $A$ is a quadratic form with coefficients depending on $p,a_1$ and $a_2$. Using ideas originating in \cite{DGL}, see also \cite{GV, Hos}, we can show that $A$ is a positive definite quadratic form if
\begin{equation}\label{EQU: GWP cond 2}
    2<p<2+2\left(r^2+2r\sqrt{1+r^2}\right)
\end{equation}
where $r=|\frac{a_1}{a_2}|$. The positivity of $A$ gives us an a priori bound of the form \eqref{EQU: a prori bound}. The conditions \eqref{EQU: GWP cond 1} and \eqref{EQU: GWP cond 2} give the restriction on the dissipation-dispersion ratio in Theorem \ref{GWP}.

The proof of Theorem \ref{Th: AS GWP} relies on an invariant measure argument first used by Bourgain, \cite{Bou94}, in the context of dispersive PDEs with random initial data. Although in this paper we apply a version of the argument in \cite{Bou94}, due to \cite{DPD}, more adapted to the setting of stochastic PDEs.

We end this introduction with a few remarks.

\begin{remark}
\textnormal{
SCGL can be viewed as an equation interpolating the parabolic SPDE setting corresponding to $a_2=c_2=0$, SQE, and the dispersive SPDE setting corresponding to $a_1=c_1=0$, a stochastic nonlinear Schr\"odinger equation (SNLS) with white noise forcing. As explained at the beginning of this subsection, SQE is well understood in dimensions one, two and three. Contrastingly, almost nothing is known about SNLS with white noise forcing. Local well-posedness is an open problem even in the one-dimensional setting. Local well posedness of SNLS with a smoothed noise has been studied in many papers, see for example \cite{FOW} where local well-posedness of SNLS with an almost space-time white noise forcing is proven.
More is understood for other dispersive SPDEs with white noise forcing. See for example \cite{Oh, GKO, GKO2}.
}
\end{remark}

\begin{remark}
\textnormal{
It should be possible to adapt the arguments in \cite{TW} or \cite{HM} to show that the transition semi-group associated to \eqref{EQU: renormalised equ for u} satisfies the strong Feller property. One should then be able to extend the almost sure GWP result of Theorem \ref{Th: AS GWP} to deterministic GWP. See for example \cite{DDF} where this was done for the one dimensional Gross-Pitaevskii equation.
}
\end{remark}

\begin{remark}
\textnormal{
Using ideas from \cite{HRW}, it should be possible to prove the \emph{Triviality} of the two-dimensional SCGL, \eqref{SCGL}. That is, if $\{u_N\}_{N\in \mathbb{N}}$ is a sequence of solutions to, the unrenormalized, \eqref{EQU: Trunc SCGL} that converges to some $u$, then $u=0$.
}
\end{remark}

\begin{remark}
\textnormal{
Using the estimates used to prove Theorem \ref{Th: AS GWP}, it should be possible to prove a coming down from infinity result like in \cite{MW2,TW}. However, as we had no use for this result, in this paper we did not pursue this.
}
\end{remark}

\section{Function spaces and basic estimates}\label{Sec: Basic ests}
Throughout the rest of this paper we use the notation $A\lesssim B $ in place of writing $A\leq C B$ where $C$ is some inessential, perhaps very large, constant. If we want to explicate what the implicit constant depends on, we use a subscript. For example, $A \lesssim_{s,p,d} B$ means $A \leq CB$ where $C$ depends on $s,p$ and $d$.

In the following, unless explicitly stated otherwise, all functions spaces are defined with $\T^2$ as the underlying space. For ease of notation we will omit writing $\T^2$ when referring to the function spaces. For example we write $L^p$ instead of $L^p(\T^2)$.

In this section we define the Besov and Sobolev spaces we work with in this paper and then collect some estimates involving these spaces. 

Let $\varphi_0$ be a smooth function satisfying
\begin{equation*}
\varphi_0(\xi)=\begin{cases}
1, & |\xi|\leq \frac{5}{8}\\
0,& |\xi|\geq \frac{8}{5}
\end{cases}
\end{equation*}
and for $j\geq 1$ define
\begin{equation*}
\varphi_j(\xi)=\varphi_0(\xi/2^j)-\varphi_0(\xi/2^{j-1}).
\end{equation*}
The function $\varphi_j$ is supported on an annuli of width approximately $2^j$. Note that
\begin{equation*}
\sum\limits_{j\geq 0}\varphi_j=1.
\end{equation*}
For a function $f:\T^2\rightarrow\C$ with Fourier coefficients $\{ \widehat{f}(n)\}_{n\in\Z}$, we then define, as Fourier multipliers, the $L^p$ projectors:
\begin{equation*}
\widehat{\delta_j f}(n)=\varphi_j(n)\widehat{f}(n).
\end{equation*}
The function $\delta_j f$ is a smooth localization of $f$ around frequencies of size approximately $2^j$. Using the $L^p$ projectors, for $s\in\R$ and $p,q\in[1,\infty]$, we define the Besov spaces $B^s_{p,q}$ through the norm
\begin{equation*}
\norm{f}_{B^s_{p,q}(\T^2)}=\norm{ 2^{js}\norm{\delta_j f}_{L^p} }_{\ell^q}.
\end{equation*}
We define the Besov-H\"older spaces, $C^s=B^s_{\infty,\infty}$.

We will now state some useful Besov space estimates. Unless otherwise mentioned, proofs of the estimates in this section can be found in \cite{MW1}. 

First we have the following Besov space embedding result.
\begin{proposition}\label{besovembed}
Let $s_0\leq s_1$, $1\leq q\leq \infty$ and $1\leq r\leq p\leq \infty$ be such that 
\begin{equation*}
s_1=s_0+d\left(\frac{1}{r}-\frac{1}{p}\right). 
\end{equation*}
Then,
\begin{equation*}
\norm{f}_{B^{s_0}_{p,q}}\lesssim \norm{f}_{B^{s_1}_{r,q}}.
\end{equation*}
\end{proposition}

\begin{proposition}\label{besovdual}
Let $s\in[0,1)$ and $p,q\in[1,\infty]$. Let $p'$ and $q'$ be the conjugate exponents of $p$ and $q$ respectively. Then,
\begin{equation*}
|\langle f,g\rangle |\lesssim \norm{f}_{B^{s}_{p,q}}\norm{g}_{B^{-s}_{p',q'}}
\end{equation*}
where $\langle\cdot,\cdot\rangle$ denotes the $L^2$ inner product.
\end{proposition}

\begin{proposition}\label{nablaest}
Let $s\in(0,1)$. Then,
\begin{equation*}
\norm{f}_{B_{1,1}^s}\lesssim \norm{f}^{1-\s}_{L^1}\norm{\nabla f}_{L^1}^\sigma+\norm{f}_{L^1}
\end{equation*}
\end{proposition}

\begin{proposition}\label{applicparaprod}
Suppose $s_0<0$ and $s_1>0$ satisfy $s_0+s_1>0$. Then the mapping $(f,g)\mapsto fg$ can be extended from to a continuous linear map from $C^{s_0}\times C^{s_1}$ to $C^{s_1}$ and
\begin{equation*}
\norm{fg}_{C^{s_0}}\lesssim \norm{f}_{C^{s_0}}\norm{g}_{C^{s_1}}.
\end{equation*}
\end{proposition}

The Besov-H\"older spaces $C^s$ have the following algebra property.
\begin{proposition}\label{PROP: algebra prop}
For $s>0$ we have
\begin{equation*}
    \norm{fg}_{C^s}\leq \norm{f}_{C^{s}}\norm{g}_{C^{s}}.
\end{equation*}
\end{proposition}

For $s\in \R$ and $1\leq p\leq \infty$ we define the Sobolev space $W^{s,p}$ through the norm
\begin{equation*}
\norm{f}_{W^{s,p}} = \norm{\langle \nabla\rangle^{s}f}_{L^p}. 
\end{equation*}
We then have the following Sobolev embedding result.
\begin{proposition}\label{PROP: Sobolev}
Suppose $s_0\leq s_1$ and $1\leq q\leq p\leq  \infty$ satisfy $s_1 = s_0 + (\frac{1}{p}-\frac{1}{q})$. Then,
\begin{equation*}
    \norm{f}_{W^{s_0, p}} \lesssim \norm{f}_{W^{s_1, q}}. 
\end{equation*}
\end{proposition}

We are mainly interested in the Sobolev space corresponding to $p=\infty$. The following proposition shows that, up to a $\varepsilon$ loss in regularity we can transfer estimates between $W^{s,\infty}$ and $C^{s}$.
\begin{proposition}\label{holdersobolevloddif}
For all $s\in \R$ and $\varepsilon>0$ we have
\begin{equation*}
   \norm{f}_{C^{s}} \lesssim \norm{f}_{W^{s,\infty}} \lesssim \norm{f}_{C^{s+\eps}}.
\end{equation*}
\end{proposition}

The following three heat-type linear smoothing estimates are used to prove WSCGL is locally well-posed. For proofs of these estimates we refer the reader to \cite{BCD}, where the results are proven for $a_2 = 0$. The proofs easily adapt to the case $a_2 \neq 0$.
\begin{proposition}\label{PROP: Heat smooth}
Let $s_0\leq s_1$. Recall $S(t) = e^{t(a_1+ia_2)[\Delta-1]}$.  Then,
\begin{equation*}
    \norm{S(t)f}_{C^{s_1}}\lesssim t^{\frac{s_0-s_1}{2}}\norm{f}_{C^{s_0}}.
\end{equation*}
\end{proposition}

\begin{proposition}\label{Prop: Time cont}
Let $s_0\leq s_1$ be such that $s_1-s_0\leq 2$. Then,
\begin{equation*}
    \norm{(1-S(t))f}_{C^{s_0}}\lesssim t^{\frac{s_1-s_0}{2}}\norm{f}_{C^{s_1}}.
\end{equation*}
\end{proposition}

The previous Proposition shows that, if $s_1>s_0$ and $f\in C^{s_1}$, then the mapping $t\mapsto S(t)f$ is continuous as a mapping from $[0,\infty)$ to $C^{s_0}$. The proposition however, says nothing about continuity if $s_0=s_1$. The following proposition states that this mapping is continuous, even though we do not have an explicit bound.

\begin{proposition}
Suppose $s_0\in \R$ and $f\in C^{s_0}$. Then the mapping $t\mapsto S(t)f$ is continuous as a mapping from $[0,\infty)$ to $C^{s_0}$.
\end{proposition}

\section{Laguerre polynomial formulae}\label{SEC: Lag forms}

This section is devoted to proving the Laguerre polynomial sum formula, \eqref{EQU: Lag sum form} and the Laguerre polynomial expectation formula \eqref{Lag expectation formula}.
\subsection{Sum formula}

The generalized Laguerre polynomials enjoy the following, classical, three point rules:
\begin{equation*}
(n+\ell)L_{n-1}^{(\ell)}(x)=nL_{n}^{(\ell)}(x)+xL_{n-1}^{(\ell+1)}(x),\quad L_{n}^{(\ell)}(x)-L_{n-1}^{(\ell)}(x)=L_{n}^{(\ell-1)}(x).
\end{equation*}
Together these relations imply 
\begin{equation}\label{Lagreccurrence}
(n+\ell)L_{n}^{(\ell-1)}(x)=\ell L_{n}^{(\ell)}(x)-xL_{n-1}^{(\ell+1)}(x).
\end{equation}
There is also a well known recurrence formula for derivatives of generalized Laguerre polynomials:
\begin{equation}\label{lag derivative}
\frac{d}{dx}L^{\ell}_{k}(x)=-L^{\ell+1}_{k-1}(x),\quad \mbox{ for } k\geq 1.
\end{equation}
The aim of this subsection is to prove the following sum formula.
\begin{lemma}\label{lagsum}
Let $m\geq 1$ and $\ell\geq 0$. Then the following is true:
\begin{align*}
(-1)^{m}m!L_{m}^{(\ell)}(|x+y|^2;\sigma)(x+y)^\ell=\sum\limits_{\substack{0\leq i\leq m+\ell\\0\leq j\leq m}} \binom{m+\ell}{i}\binom{m}{j}P^{m,\ell,\sigma}_{i,j}(y,\overline{y})x^i\overline{x}^j
\end{align*}
where
\begin{equation*}
P_{i,j}^{m,\ell,\sigma}(y,\overline{y})=\begin{cases}
(-1)^{m-j}(m-j)!L^{(\ell+j-i)}_{m-j}(|y|^2;\sigma)y^{\ell+j-i},& \ell+j\geq i\\
\\
(-1)^{m+\ell-i}(m+\ell-i)!L^{(i-j-\ell)}_{m+\ell-i}(|y|^2;\sigma)\overline{y}^{i-j-\ell},&\ell+j\leq i.
\end{cases}
\end{equation*}
\end{lemma}
\begin{proof}
We view both sides of the equality as polynomials in $x$ and $\overline{x}$ with coefficients depending on $y$ and $\overline{y}$. By scaling it suffices to prove the lemma for $\sigma=1$.

Note that the lemma is true for all $m,\ell\geq 0$ satisfying $2m+\ell\leq 1$. To prove the Lemma for all $m$ and $\ell$ we induct on $2m+\ell$.  

Let $n\in\N$ and suppose the statement in the Lemma is true for all $m,\ell\geq 0$ satisfying $2m-\ell<n$. Then for $m,\ell\geq 0$ such that $2m+\ell=n$, using \eqref{lag derivative} and the inductive hypothesis, we have,
\begin{align*}
\begin{split}
\frac{\partial}{\partial\overline{x}}\left[(-1)^m m!L^{(\ell)}_{m}(|x+y|^2)(x+y)^\ell \right]&=(-1)^{m-1}m!L^{(\ell+1)}_{m-1}(|x+y|^2)(x+y)^{\ell+1}\\
&=\sum\limits_{\substack{0\leq i\leq m+\ell\\0\leq j\leq m-1}} m\binom{m+\ell}{i}\binom{m-1}{j}P^{m-1,\ell+1,\sigma}_{i,j}(y,\overline{y})x^i\overline{x}^j.
\end{split}
\end{align*}
``Partially integrating'' this expression with respect to $\overline{x}$ we get
\begin{align}
(-1)^m m!&L^{(\ell)}_{m}(|x+y|^2)(x+y)^\ell\nonumber\\
&=\sum\limits_{\substack{0\leq i\leq m+l\\0\leq j\leq m-1}} \frac{m}{j+1}\binom{m+l}{i}\binom{m-1}{j}P^{m-1,\ell+1,\sigma}_{i,j}(y,\overline{y})x^i\overline{x}^{j+1}+C(y,\overline{y},x)\nonumber\\
&=\sum\limits_{\substack{0\leq i\leq m+\ell\\1\leq j\leq m}} \binom{m+\ell}{i}\binom{m}{j}P^{m,\ell,\sigma}_{i,j}(y,\overline{y})x^i\overline{x}^{j}+C(y,\overline{y},x)\label{EQU: after partial int}
\end{align}
where the last equality comes from relabeling $j$ in the summation and noting $P^{m-1,\ell+1,\sigma}_{i,j-1}(y,\overline{y})=P^{m,\ell,\sigma}_{i,j}(y,\overline{y})$. Differentiating the left hand side of \eqref{EQU: after partial int} with respect to $x$ and using the three point rule  \eqref{Lagreccurrence} we have,
\begin{align*}
\frac{\partial}{\partial x}\left[ (-1)^m m!L^{(\ell)}_{m}(|x+y|^2)(x+y)^\ell  \right]&=-(-1)^m m!L^{(\ell+1)}_{m-1}(|x+y|^2)(x+y)^{\ell-1}|x+y|^2\\
&+\ell(-1)^m m!L^{(\ell)}_{m}(|x+y|^2)(x+y)^{\ell-1}\\
&=(m+\ell)(-1)^{m}m!L^{(\ell-1)}_{m}(|x+y|^2)(x+y)^{\ell-1}.
\end{align*}
Equating this with the derivative of the right hand side of \eqref{EQU: after partial int},
\begin{align*}
\sum\limits_{\substack{0\leq i\leq m+\ell-1\\0\leq j\leq m}} &\binom{m+\ell-1}{i}\binom{m}{j}P^{m\ell-1}_{i,j}(y,\overline{y})x^i\overline{x}^{j}\\
&=\sum\limits_{\substack{1\leq i\leq m+\ell\\1\leq j\leq m}}\frac{i}{m+\ell} \binom{m+\ell}{i}\binom{m}{j}P^{m,\ell}_{i,j}(y,\overline{y})x^{i-1}\overline{x}^{j}+\frac{\partial}{\partial x}\ C(y,\overline{y},x)\\
&=\sum\limits_{\substack{0\leq i\leq m+\ell-1\\1\leq j\leq m}} \binom{m+\ell-1}{i}\binom{m}{j}P^{m,\ell-1}_{i,j}(y,\overline{y})x^{i}\overline{x}^{j}+\frac{1}{m+\ell}\frac{\partial}{\partial x}\ C(y,\overline{y},x)
\end{align*}
where in the last equality we used the fact that $P^{m,\ell-1}_{i,j}(y,\overline{y})=P^{m,\ell}_{i+1,j}(y,\overline{y})$. Hence we get an expression for $\frac{\partial}{\partial x}\ C(y,\overline{y},x)$,
\begin{equation*}
\frac{\partial}{\partial x}\ C(y,\overline{y},x)=(m+\ell)\sum\limits_{0\leq i\leq m+\ell-1} \binom{m+\ell-1}{i}P^{m,\ell-1}_{i,0}(y,\overline{y})x^i.
\end{equation*}
``Partially integrating'' this expression we get
\begin{align*}
C(y,\overline{y},x)&=\sum\limits_{0\leq i\leq m+\ell-1} \frac{m+1}{i+1}\binom{m+\ell-1}{i}P^{m,\ell-1}_{i,0}(y,\overline{y})x^{i+1}+C(y,\overline{y})\\
&=\sum\limits_{1\leq i\leq m+\ell} \binom{m+\ell}{i}P^{m,\ell}_{i,0}(y,\overline{y})x^{i}+C(y,\overline{y})
\end{align*}
where we have relabeled the sum in the second inequality.
This shows,
\begin{equation}\label{EQU: 3.4}
(-1)^{m}m!L_{m}^{(\ell)}(|x+y|^2;\sigma)(x+y)^\ell=\sum\limits_{\substack{0\leq i\leq m+\ell\\0\leq j\leq m\\(i,j)\neq(0,0)}} \binom{m+\ell}{i}\binom{m}{j}P^{m,\ell,\sigma}_{i,j}(y,\overline{y})x^i\overline{x}^j+C(y,\overline{y}).
\end{equation}
For $C(y,\overline{y})$, note that when $x=0$ \eqref{EQU: 3.4} reduces to,
\begin{equation*}
(-1)^{m}m!L_{m}^{(\ell)}(|y|^2;\sigma)y^\ell=C(y,\overline{y}).
\end{equation*}
This completes the proof.
\end{proof}

\subsection{Expectation formula}
The generalized Hermite polynomials satisfy the following recurrence relation
\begin{equation}\label{Hermrec}
H_{k+1}(x;\sigma)=x H_k(x;\sigma)-\sigma H_{k-1}(x;\sigma).
\end{equation}
These polynomials also enjoy the following properties.
\begin{proposition}
Let $k\geq 0$ and $\sigma,\beta\in \R$. Then the following are true.
\begin{enumerate}[label=(\roman*)]
\item\begin{equation}\label{ukexp}
\int_\R H_k(x;\sigma)e^{ux-\frac{x^2}{2}}dx=\sqrt{2\pi}H_k(u;\sigma-1)e^\frac{u^2}{2},
\end{equation}

\item \begin{equation}\label{EQU: imag hermite}
i^kH_k(x;-\sigma)=H_k(ix;\sigma)
\end{equation}

\item\begin{equation}\label{EQU: Hermite gen sum}
H_n(x;\ell+\beta)=\sum\limits_{k=0}^n\binom{n}{k}H_k(x;\ell) H_{n-k}(x;\beta),
\end{equation}
\item if $\sigma>0$,
\begin{equation*}
H_k(x;\sigma)=\sigma^{k/2}H_k(x/\sqrt{\sigma}).
\end{equation*}
\end{enumerate}
\end{proposition}

\begin{proof}
These facts can be proven by using the recurrence relation \eqref{Hermrec} and a standard induction argument. We will just prove (1) to give the reader a taste of how to complete such an argument. Note that the generating function for the Hermite polynomials is
\begin{equation*}
e^{xu-\frac{u^2}{2}}=\sum\limits_{n=0}^\infty\frac{H_n(x)}{n!}u^n.
\end{equation*}
Multiplying both sides by $H_n(x)e^{\frac{x^2}{2}}$ and using the fact that $\int_\R H_n(x)H_m(x)e^{\frac{x^2}{2}}\,dx=\sqrt{2\pi}n!\sigma_{nm}$,
\begin{equation*}
e^\frac{-u^2}{2}\int_\R e^{xu-\frac{x^2}{2}}H_n(x)dx=\sqrt{2\pi}u^n
\end{equation*}
and so the result is true for $\sigma=1$. As $H_0(x;\sigma)=1$ and $H_1(x;\sigma)=x$ for all $\sigma$, the $k=0$ and $k=1$ cases are also true. From the recurrence relation \eqref{Hermrec}
the result is then true for all $\sigma$.
\end{proof}

The key ingredient in the real valued analogue of Proposition \ref{PROP:sconv} in the next section is the following well known identity:
\begin{equation}\label{hermiteortho}
\E[H_k(f;\sigma_f)H_\ell(g;\sigma_g)]=k!\delta_{k\ell}
\end{equation}
where $f$ and $g$ are Gaussian random variables with variances $\sigma_f$ and $\sigma_g$ respectively. To prove Proposition \ref{PROP:sconv} we need the following Laguerre polynomial analogue of \eqref{hermiteortho}.

\begin{proposition}\label{orthoglag}
Let $f$ and $g$ be mean-zero complex valued Gaussian random variables with variances $\sigma_f$ and $\sigma_g$ respectively. Then,
\begin{equation*}
\E\left[L^{(\ell)}_k(|f|^2;\sigma_f)f^\ell\overline{L^{(\ell)}_m(|g|^2;\sigma_g)g^\ell} \right]=\delta_{km}\frac{(k+\ell)!}{k!}\left| \E[f\overline{g}]  \right|^{2k}\E[f\overline{g}]^\ell.
\end{equation*}
\end{proposition}
The above proposition was proven in \cite{OT1} for $\ell=0$ and $\ell=1$. The proof for the general case proved in this section is the natural generalization of the proof in \cite{OT1}. We use the following elementary lemma.
\begin{lemma}\label{Requarter}
Let $g$ be a mean-zero complex valued random variable. Then
\begin{equation*}
\E\left[e^{\Re g} \right]=e^{\frac{1}{4}\E[|g|^2]}.
\end{equation*}
\end{lemma}

\begin{proof}[Proof of Proposition~\ref{PROP:sconv}]
It suffices to prove the Lemma assuming $\sigma_f=\sigma_g=1$. Let $f_1=\Re f$ and $f_2=\Im f$. Using the binomial expansion formula for $(f_1+if_2)^\ell$ and then applying \eqref{ukexp} with $\sigma=1$, $u=\sqrt{\frac{-2t}{1-t}}f_1$ and again with $\sigma=1$, $u=\sqrt{\frac{-2t}{1-t}}f_2$,
\begin{align}
G_{\ell}&(t,|f|^2)f^\ell=\frac{1}{(1-t)^{\ell+1}}(f_1+if_2)^\ell e^{\frac{-t}{1-t}(f_1^2+f_2^2)}\label{EQU: single Gl}\\
&=\sum\limits_{k=0}^l\frac{1}{(1-t)^{\ell+1}}\binom{\ell}{k}i^{\ell-k}f_1^kf_2^{\ell-k}e^{\frac{-t}{1-t}(f_1^2+f_2^2)}\nonumber\\
&=\sum\limits_{k=0}^\ell\binom{\ell}{k}\frac{i^{\ell-k}}{(\sqrt{-2t})^l(1-t)^{l/2+1}}\frac{1}{2\pi}\int_{\R^2}H_k(x_1)H_{\ell-k}(x_2)e^{-\frac{x_1^2+x_2^2}{2}}e^{\sqrt{\frac{-2t}{1-t}}(x_1f_1+x_xf_2) }dx_1dx_2.\nonumber
\end{align}

Given $x_1,x_2,x_3,x_4\in\R$, we set $x=x_1+ix_2$ and $y=y_1+iy_2$. For $s,t\in (-1,0)$ applying \eqref{EQU: single Gl} twice and taking an expectation gives, 
\begin{align*}
\int_{\Omega}G_{\ell}(t,|f|^2)f^\ell &\overline{G_{\ell}(t,|g|^2)g^\ell} dP(\omega)\\
&=\sum\limits_{k,m=0}^\ell\binom{\ell}{k}\binom{\ell}{m}i^{2\ell-k-m}\frac{1}{(\sqrt{-2t})^l(1-t)^{l/2+1}}\frac{1}{(\sqrt{-2s})^\ell(1-s)^{l/2+1}}\frac{1}{4\pi^2}\\
&\hphantom{xx}\times\int_{\R^4}H_{k}(x_1)H_{\ell-k}(x_2)H_{m}(y_1)H_{\ell-m}(y_2)e^{-\frac{|x|^2+|y|^2}{2}}\\
&\hphantom{xx}\times\int_{\Omega}\exp\left(\Re\left(\sqrt{\frac{-2t}{1-t}}\overline{x}f+\sqrt{\frac{-2s}{1-s}}y\overline{g} \right)  \right)dx_1dx_2dy_1dy_2\\
&=\sum\limits_{k,m=0}^\ell\binom{\ell}{k}\binom{\ell}{m}i^{2\ell-k-m}\frac{1}{(\sqrt{-2t})^\ell(1-t)^{l/2+1}}\frac{1}{(\sqrt{-2s})^\ell(1-s)^{\ell/2+1}}\frac{1}{4\pi^2}\\
&\hphantom{xx}\times\int_{\R^4}H_{k}(x_1)H_{\ell-k}(x_2)H_{m}(y_1)H_{\ell-m}(y_2)\\
&\hphantom{xx}\times e^{-\frac{|x|^2}{2(1-t)}-\frac{|y|^2}{2(1-s)}}e^{\frac{1}{2}\Re\left(\sqrt{\frac{-2t}{1-t}}\sqrt{\frac{-2t}{1-t}}\overline{x}y\E[f\overline{g}]  \right)}dx_1dx_2dy_1dy_2
\end{align*}
where in the second inequality we used Lemma \ref{Requarter}. Applying the change of variables $x,=\frac{1}{\sqrt{1-t}}x$ and $y,=\frac{1}{\sqrt{1-s}}y$ and then using Lemma \ref{ukexp} with $u =\sqrt{ts}\Re(y\E[f\overline{g}])$ and again with $u =\sqrt{ts}\Im(y\E[f\overline{g}])$ we have,
\begin{align*}
\int_{\Omega}&G_{\ell}(t,|f|^2)f^\ell \overline{G_{\ell}(t,|g|^2)g^\ell} dP(\omega)\\
&=\sum\limits_{k,m=0}^\ell\binom{\ell}{k}\binom{\ell}{m}i^{2\ell-k-m}\frac{1}{(2ts)^{\ell/2}}\frac{1}{4\pi^2}\\
&\hphantom{xxx}\times\int_{\R^4}H_{k}(x_1;(1-t)^{-1})H_{\ell-k}(x_2;(1-t)^{-1})H_{m}(y_1;(1-s)^{-1})H_{\ell-m}(y_2;(1-s)^{-1})\\
&\hphantom{xxx}\times e^{-\frac{|x|^2}{2}-\frac{|y|^2}{2}}e^{\sqrt{ts}x_1\Re\left(y\E[f\overline{g}]\right)+\sqrt{ts}x_2\Im\left(y\E[f\overline{g}]\right)}dx_1dx_2dy_1dy_2\\
=&\sum\limits_{k,m=0}^\ell\binom{\ell}{k}\binom{\ell}{m}\frac{i^{2\ell-k-m}}{2\pi(2ts)^{\ell/2}}\int_{\R^2}H_{k}(\sqrt{ts}\Re\left(y\E[f\overline{g}]\right);\tfrac{t}{1-t})H_{\ell-k}(\sqrt{ts}\Im\left(y\E[f\overline{g}]\right);\tfrac{t}{1-t})\\
&\hphantom{xxx}\times H_{m}(y_1;(1-s)^{-1})H_{\ell-m}(y_2;(1-s)^{-1})e^{-\frac{|y|^2}{2}}e^{\tfrac{1}{2}\sqrt{ts}|y|^2|\E[f\overline{g}]|^2}dy_1dy_2\\
&=\frac{1}{(2ts)^{\ell/2}}\frac{1}{2\pi}\int_{\R^2}\left(\sum\limits_{k=0}^\ell \binom{\ell}{k} H_{k}(\sqrt{ts}\Re\left(y\E[f\overline{g}]\right);\tfrac{t}{1-t})i^{\ell-k}H_{\ell-k}(\sqrt{ts}\Im\left(y\E[f\overline{g}]\right);\tfrac{t}{1-t})\right)\\
&\hphantom{xxx}\times\left(\sum\limits^\ell_{m=0}\binom{\ell}{m}H_{m}(y_1;(1-s)^{-1})i^{\ell-m}H_{\ell-m}(y_2;(1-s)^{-1})\right)e^{-\tfrac{1}{2}\left(1-\sqrt{ts}|\E[f\overline{g}]|^2  \right)|y|^2}dy_1dy_2.
\end{align*}
Using \eqref{EQU: imag hermite} and then \eqref{EQU: Hermite gen sum} we get,
\begin{align*}
\int_{\Omega}&G_{\ell}(t,|f|^2)f^\ell \overline{G_{\ell}(t,|g|^2)g^\ell} dP(\omega)=\frac{\E[f\overline{g}]^\ell}{2^{\ell+1}\pi}\int_{\R^2}|y|^{2\ell}e^{-\tfrac{1}{2}\left(1-\sqrt{ts}|\E[f\overline{g}]|^2  \right)|y|^2}dy.
\end{align*}
Integrating the right hand side over $\R^2$ using the formula
\begin{equation*}
    \int_{\R^2} |y|^{2\ell}e^{-\beta |y|^2}dy = \frac{\ell! \pi}{\beta^{\ell+1}}
\end{equation*}
we get,
\begin{align*}
\begin{split}
\int_{\Omega}G_{\ell}(t,|f|^2)f^\ell \overline{G_{\ell}(t,|g|^2)g^\ell} dP(\omega)&=\frac{\E[f\overline{g}]^\ell}{2^{\ell+1}\pi}\frac{2^{\ell+1}\pi \ell!}{\left(1-ts|\E[f\overline{g}]|^2\right)^{\ell+1}}\\
&=\ell!\frac{\E[f\overline{g}]^\ell}{\left(1-ts|\E[f\overline{g}]|^2\right)^{\ell+1} }\\
&=\sum_{k=0}^\infty \ell!\binom{\ell+k}{\ell}t^ks^k|\E[f\overline{g}]|^{2k}\E[f\overline{g}]^\ell
\end{split}
\end{align*}
where the last equality uses the Maclaurin series
\begin{equation*}
\frac{1}{(1-x)^{\ell+1}}=\sum\limits_{n=0}^\infty \binom{\ell+n}{\ell}x^n
\end{equation*}
 which is valid for $|x|\leq 1$. On the other hand, from the the generating function $G_\ell$ we have,
\begin{align*}
\int_{\Omega}G_{\ell}(t,|f|^2)f^\ell \overline{G_{\ell}(t,|g|^2)g^\ell} dP(\omega)&=\sum\limits_{k,m=0}^\infty t^ns^m\E\left[L^{(\ell)}_k(|f|^2;\sigma_f)f^\ell\overline{L^{(\ell)}_m(|g|^2;\sigma_g)g^\ell} \right].
\end{align*}
The proposition follows by comparing coefficients.
\end{proof}

\section{On the stochastic convolution}\label{SEC: Sto Conv}

In this section we will give a proof of Proposition \ref{PROP:sconv}, establishing regularity estimates for the Wick ordered powers $:\Psi_N^{k}\overline{\Psi_N}^\ell:$. Before we do this, we state a version of the well known Wiener chaos estimate that will be used extensively in this section.

\begin{proposition}\label{PROP: Weiner chaos}
Let $\{g_n\}_{n\in \Z}$ be a sequence of standard independent identically distributed Gaussian random variables. Let $k\in \mathbb{N}$ and let $\{P_j(\overline{g})\}_{j\in \mathbb{N}}$ be a sequence of polynomials in $\{g_n\}_{n\in \Z}$ of degree at most $k$. Then for $p\geq 2$,
\begin{equation*}
    \bigg\lVert {\sum\limits_{j\in \mathbb{N}}P_j(\overline{g})}\bigg\rVert _{L^p(\Omega)} \leq (p-1)^\frac{k}{2}\bigg\lVert {\sum\limits_{j\in \mathbb{N}}P_j(\overline{g})}\bigg\rVert_{L^2(\Omega)}.
\end{equation*}
\end{proposition}
For a proof of this result, see \cite{Simon}. 

\begin{proof}[Proof of Proposition~\ref{PROP:sconv}]
We prove this proposition using a Fourier analytic approach similar to that in \cite[Proposition 2.1]{GKO}.

We will assume $k\geq \l$, the other case is similar. In the estimates in this proof, for ease of notation, we will ignore the the terms $\chi_N(n)$ in the definition of $\Psi_N$, \eqref{conv} as these terms can simply be bounded by $1$, independently of $N$ and $n$. By Lemma \ref{holdersobolevloddif} it suffices to prove the proposition with $W^{-\eps,\infty}$ in place of $C^{-\eps}$. Further, as $L^{p_1}(\Omega)\subset L^{p_2}(\Omega)$ for $p_1\leq p_2$, it suffices to prove the proposition for $p$ sufficiently large. 

First we drive a useful formula used throughout this proof. For $t_1\leq t_2$, by the independence of $\beta_n$ and $\beta_m$ for $m\neq n$, the independent increment property of Brownian motion and the It\^o isometry we have
\begin{align}
\E\big[ \Psi_N(x,t_1)\overline{\Psi_N(y,t_2)} \big]
&=\gamma\sum\limits_{|n|,|m|\leq N} e_n(x)e_{-m}(y)\E\bigg[\int_{-\infty}^{t_1}e^{-(t_1-t')(a_1+ia_2)(|n|^2+1)}d\beta_n(t')\nonumber\\
&\hspace{130pt}\times\overline{\int_{-\infty}^{t_2}e^{-(t_2-t')(a_1+ia_2)(|m|^2+1)}d\beta_m(t')}  \bigg]   \nonumber\\
&=\gamma\sum\limits_{0\leq |n|\leq N} e_n(x-y)\E\bigg[\int_{-\infty}^{t_1}e^{-(t_1-t')(a_1+ia_2)(|n|^2+1)}d\beta_n(t')\nonumber\\
&\hspace{130pt}\times\overline{\int_{-\infty}^{t_1}e^{-(t_2-t')(a_1+ia_2)(|n|^2+1)}d\beta_n(t')}  \bigg]   \nonumber\\
&=2\gamma\sum\limits_{0\leq |n|\leq N} e_n(x-y)e^{-(t_2-t_1)(a_1-ia_2)(|n|^2+1)}\int_{-\infty}^{t_1}e^{-2(t_1-t')(|n|^2+1)a_1}dt'\nonumber\\
&=\sum\limits_{0\leq |n|\leq N} e_n(x-y)e^{-(t_2-t_1)(a_1-ia_2)(|n|^2+1)}\frac{\gamma}{a_1(|n|^2+1)}\nonumber\\
&= \sum\limits_{0\leq |n|\leq N} e_n(x-y)\zeta(n,t_1,t_2)\label{EQU: Simple Expect xy}
\end{align}
where 
\begin{equation*}
   \zeta(n, t_1, t_2)=e^{-(t_2-t_1)(a_1-ia_2)(|n|^2+1)}\frac{\gamma}{a_1(|n|^2+1)}.  
\end{equation*}
Note that 
\begin{equation}
    |\zeta(n,t_1,t_2)|\lesssim_{a_1,\gamma}\langle n\rangle^{-2}\label{EQU: Zeta bound}.
\end{equation}
When $t_1=t_2$,  $\zeta_n(n,t_1,t_2)$ is independent of $t_1$ and $t_2$ and so we write $\zeta(n)$ instead of $\zeta_n(n,t_1,t_2)$.

Now we will show that $\Psi_N(\cdot, t)\in W^{-\eps, \infty}$ for a fixed $t$. Applying the Bessel potentials $\langle \nabla_x\rangle^{-\eps}$ and $\langle \nabla_y\rangle^{-\eps}$ to \eqref{EQU: Simple Expect xy} with $t_1=t_2=t$ we have,
\begin{align*}
\E[|\langle \nabla_x\rangle^{-\eps}\Psi_N(x,t)|^2]=\sum\limits_{|n|\leq N} \frac{\zeta(n)}{\langle n\rangle^{2\eps}}\lesssim \sum\limits_{|n|\leq N} \frac{1}{\langle n\rangle^{2+2\eps}}\lesssim 1< \infty
\end{align*}
uniformly in $N\in \N$, $x\in \T^2$ and $t\in\R$. Using Proposition \ref{PROP: Sobolev}, switching the order of integration and then using Proposition \ref{PROP: Weiner chaos} we have
\begin{align}
\begin{split}\label{EQU: Sobolev + Weiner Chaos}
\E\left[\norm{\Psi_N(t,\cdot)}_{W^{-\eps,\infty}}^p  \right]& \lesssim_{p,\eps}  \E\left[\norm{\Psi_N(t,\cdot)}_{W^{-\eps/2,p}}^p  \right]\\
&=\int_{\T^2}\E[|\langle\nabla\rangle^{-\eps/2}\Psi_N(t,x)|^p ]dx\\
&\lesssim \int_{\T^2}\E[|\langle\nabla\rangle^{-\eps/2}\Psi_N(t,x)|^2]dx \\
&\lesssim 1.
\end{split}
\end{align}

Now we will show that $:\!\Psi_N(\cdot,t)^k\overline{\Psi_N(\cdot,t)}^\l\!:$ is in $W^{-\eps, \infty}$ for a fixed $t$. Using Lemma \ref{orthoglag},
\begin{align}
\begin{split}\label{EQU: app of ortholag}
\E\bigg[:\!\Psi_N(x,t)^k\overline{\Psi_N(x,t)}^\l\!:&\overline{:\!\Psi_Nk(y,t)^k\overline{\Psi_N(y,t)}^\l\!:}\bigg]\\[5pt]
&=C_{k,l}\E\big[\Psi_N(x,t)\overline{\Psi_N(x,t)}\big]^k\overline{\E\big[\Psi_N(x,t)\overline{\Psi_N(x,t)}\big]}^{\l}\\[5pt]
&=C_{k,l}\bigg(\sum\limits_{|n|\leq N} e_n(x-y)\zeta(n)\bigg)^{k}\bigg(\overline{\sum\limits_{|n|\leq N} e_n(x-y)\zeta(n)}\bigg)^{\l}\\
&=C_{k,l}\sum\limits_{|n_1|,\dots,|n_{k+\l}|\leq N} e_{n_1\dots+n_{k+\l}}(x-y)  \prod\limits_{j=1}^{k+\l} \zeta(n_j)
\end{split}
\end{align}
for some inessential constant $C_{k,\l}$. Applying the Bessel potentials $\langle \nabla_x\rangle^{-\eps}$ and $\langle \nabla_y\rangle^{-\eps}$ and then setting $x=y$ we get,
\begin{align}
\E\left[|\langle \nabla_x\rangle^{-\eps}\!:\!\Psi_Nk(x,t)^k\overline{\Psi_N(x,t)}^\l\!:|^2  \right]&\lesssim \sum\limits_{|n_1|,\dots,|n_{k+\l}|\leq N} \frac{1}{\langle n_1\dots+n_{k+\l}\rangle^{2\eps}} \prod\limits_{j=1}^{k+\l} \frac{1}{\langle n_j\rangle^2}\nonumber.
\end{align}

We want to use an argument similar to that in equation \eqref{EQU: Sobolev + Weiner Chaos} but before we can do this we need to show the sum in the above equation is bounded independently of $N$. To do this we argue by induction. Note that it is obviously bounded when $k+\l=1$. When $k+\l>1$, we split the sum into two regions corresponding to 
\begin{equation*}
\langle n_1+\cdots+n_{k+\l}\rangle\leq \langle n_{k+\l}\rangle\quad  \textnormal{  and  }\quad \langle n_1+\cdots+n_{k+\l}\rangle > \langle n_{k+\l}\rangle.
\end{equation*} 
This is motivated by the fact that 
\begin{equation*}
    \langle n_1+\cdots+n_{k+\l-1}\rangle \lesssim \max(\langle n_1+\cdots+n_{k+\l}\rangle, \langle n_{k+\l}\rangle).
\end{equation*}
With this splitting we have,
\begin{align}
   \sum\limits_{|n_{k+\l}|\leq N} &\frac{1}{\langle n_1+\dots+n_{k+\l}\rangle^{2\eps}}\frac{1}{\langle n_{k+\l}\rangle^2} \lesssim \label{EQU: Est of sum}\\ 
   &\hspace{35pt}\frac{1}{\langle n_1+\cdots+n_{k+\l-1}\rangle^{\eps}} \sum\limits_{|n_{k+\l}|\leq N} \frac{1}{\langle n_1+\cdots+n_{k+\l}\rangle^{2\eps}\langle n_{k+\l}\rangle^{2-\eps} }\nonumber\\
   & \hspace{35pt}+ \frac{1}{\langle n_1+\cdots+n_{k+\l-1}\rangle^{\eps}} \sum\limits_{|n_{k+\l}|\leq N} \frac{1}{\langle n_1+\cdots+n_{k+\l}\rangle^{\eps}\langle n_{k+\l}\rangle^{2} }\nonumber\\
   &\hspace{130pt}\lesssim \frac{1}{\langle n_1+\cdots+n_{k+\l-1}\rangle^{\eps}}\nonumber
\end{align}
independently of $N$. Hence,
\begin{align*}
    \sum\limits_{|n_1|,\dots,|n_{k+\l}|\leq N} \frac{1}{\langle n_1\dots+n_{k+\l}\rangle^{2\eps}} \prod\limits_{j=1}^{k+\l} \frac{1}{\langle n_j\rangle^2}\lesssim \sum\limits_{|n_1|,\dots,|n_{k+\l-1}|\leq N} \frac{1}{\langle n_1\dots+n_{k+\l-1}\rangle^{\eps}} \prod\limits_{j=1}^{k+\l-1} \frac{1}{\langle n_j\rangle^2}
\end{align*}
and so the desired bound follows by induction.

We have shown, 
\begin{align*}
\E\left[|\langle \nabla_x\rangle^{-\eps}\!:\!\Psi_Nk(x,t)^k\overline{\Psi_N(x,t)}^\l\!:|^2  \right] < \infty
\end{align*}
independently of $N$. Using the Propositions \ref{PROP: Sobolev} and \ref{PROP: Weiner chaos} in similar way to \eqref{EQU: Sobolev + Weiner Chaos},
\begin{align*}
\begin{split}
\E [ \norm{:\!\Psi_N(\cdot,t)^k\overline{\Psi_N(\cdot,t)}^\l\!\!:}_{W^{-\eps,\infty}}^p ] &\lesssim_{p,\eps} \E [ \norm{:\!\Psi_N(\cdot,t)^k\overline{\Psi_N(\cdot,t)}^\l\!\!:}_{W^{-\eps/2,p}}^p ]\\
&=\int_{\T^2} \E[|\langle \nabla\rangle^{-\eps/2}:\!\Psi_N(\cdot,t)^k\overline{\Psi_N(\cdot,t)}^\l\!\!:|^p]dx\\
&\lesssim\int_{\T^2} \E[|\langle \nabla\rangle^{-\eps/2}:\!\Psi_N(\cdot,t)^k\overline{\Psi_N(\cdot,t)}^\l\!\!:(\cdot,t)|^2]dx\\
&\lesssim1<\infty
\end{split}
\end{align*}
which shows $:\!\Psi_N^k\overline{\Psi_N}^\l\!\!:\,\in L^p(\Omega; L^\infty([0,T];W^{-\eps,\infty}))$ uniformly in $N$.

Now we show that $:\!\Psi_N^k\overline{\Psi_N}^\l\!\!:$ is Cauchy in $L^p(\Omega,L^\infty([0,T],W^{-\eps,\infty}))$. For $N\geq M\geq 1$, similar to \eqref{EQU: app of ortholag} we have,

\begin{align}
\E\bigg[(:\!\Psi_N(x,t)^k\overline{\Psi_N(x,t)}^\l\!:&-:\!\Psi_M(x,t)^k\overline{\Psi_M(x,t)}^\l\!:)  \nonumber\\
&\hspace{-100pt}\times\overline{(:\!\Psi_N(y,t)^k\overline{\Psi_N(y,t)}^\l\!:-:\!\Psi_M(y,t)^k\overline{\Psi_M(y,t)}^\l\!:)} \bigg]\nonumber\\
%%%%
&=C_{k,l}\E\big[\Psi_N(x,t)\overline{\Psi_N(y,t)}\big]^{k}\overline{\E\big[\Psi_N(x,t)\overline{\Psi_N(y,t)}\big]}^{\l}\nonumber\\[5pt]
&\hphantom{=}-C_{k,l}\E\big[\Psi_N(x,t)\overline{\Psi_M(y,t)}\big]^{k}\overline{\E[\Psi_N(x,t)\overline{\Psi_M(y,t)}]}^{\l}\nonumber\\[5pt]
&\hphantom{=}-C_{k,l}\E\big[\Psi_M(x,t)\overline{\Psi_N(y,t)}\big]^{k}\overline{\E\big[\Psi_M(x,t)\overline{\Psi_N(y,t)}\big]}^{\l}\nonumber\\[5pt]
&\hphantom{=}+C_{k,l}\E\big[\Psi_M(x,t)\overline{\Psi_M(y,t)}\big]^{k}\overline{\E\big[\Psi_M(x,t)\overline{\Psi_M(y,t)}\big]}^{\l}\nonumber\\[5pt]
&=C_{k,l}\bigg(\sum\limits_{|n|\leq N} e_n(x-y)\zeta(n)\bigg)^{k}\bigg(\overline{\sum\limits_{|n|\leq N} e_n(x-y)\zeta(n)}\bigg)^{\l}\nonumber\\
&\hphantom{=}-C_{k,\l}\bigg(\sum\limits_{|n|\leq M} e_n(x-y)\zeta(n)\bigg)^{k}\bigg(\overline{\sum\limits_{|n|\leq M} e_n(x-y)\zeta(n)}\bigg)^{\l}\nonumber\\
&=C_{k,l}\sum\limits_{|n_1|,\dots,|n_{k+\l}|\leq N} e_{n_1\dots+n_{k+\l}}(x-y)  \prod\limits_{j=1}^{k+\l}\zeta(n_j)\nonumber\\
& \hphantom{=}- C_{k,l}\sum\limits_{|n_1|,\dots,|n_{k+\l}|\leq M} e_{n_1\dots+n_{k+\l}}(x-y)  \prod\limits_{j=1}^{k+\l}\zeta(n_j). \label{EQU: appl of ortholag 2}
\end{align}

Using the notation 
\begin{equation*}
    \Gamma_{N,M}(\overline{n}) = \{ |n_1|,\dots,|n_{k+\l}|\leq N : |n_j| > M \textnormal{ for some } j  \}
\end{equation*}
we have
\begin{equation*}
    \textnormal{LHS of } \eqref{EQU: appl of ortholag 2} = C_{k,l}\sum\limits_{\Gamma_{N,M}(\overline{n})} e_{n_1\dots+n_{k+\l}}(x-y)  \prod\limits_{j=1}^{k+\l}\zeta(n_j).
\end{equation*}
Applying the Bessel potentials $\langle \nabla_x\rangle^{-\eps}$ and $\langle \nabla_y\rangle^{-\eps}$ and then setting $x=y$ we get
\begin{align*}
    \E\big[\big|\langle \nabla_x \rangle^{-\eps}\big(:\!\Psi_N(x,t)^k\overline{\Psi_N(x,t)}^\l\!:&-:\!\Psi_M(x,t)^k\overline{\Psi_M(x,t)}^\l\!:\big)\big|^2 \big]\\
    &= C_{k,l}\sum\limits_{\Gamma_{N,M}(\overline{n})} \frac{1}{\langle n_1\dots+n_{k+\l}\rangle^{2\eps}}  \prod\limits_{j=1}^{k+\l}\zeta(n_j).
\end{align*}
We can estimate this sum in a way similar to \eqref{EQU: Est of sum}. Indeed, without loss of generality we can assume $|n_{k+\l}|>M$. Then, adapting the estimate in \eqref{EQU: Est of sum}, we have,
\begin{align}
   \sum\limits_{N<|n_{k+\l}|\leq M} &\frac{1}{\langle n_1+\dots+n_{k+\l}\rangle^{2\eps}}\frac{1}{\langle n_{k+\l}\rangle^2} \lesssim \label{EQU: Est of sum 2}\\ 
   &\hspace{35pt}\frac{1}{\langle n_1+\cdots+n_{k+\l-1}\rangle^{\eps}} \sum\limits_{N<|n_{k+\l}|\leq M} \frac{1}{\langle n_1+\cdots+n_{k+\l}\rangle^{2\eps}\langle n_{k+\l}\rangle^{2-\eps} }\nonumber\\
   & \hspace{20pt}+ \frac{1}{\langle n_1+\cdots+n_{k+\l-1}\rangle^{\eps}} \sum\limits_{N<|n_{k+\l}|\leq M} \frac{1}{\langle n_1+\cdots+n_{k+\l}\rangle^{\eps}\langle n_{k+\l}\rangle^{2} }\nonumber\\
   &\hspace{133pt}\lesssim \frac{1}{\langle n_1+\cdots+n_{k+\l-1}\rangle^{\eps}}\frac{1}{M^\frac{\eps}{2}  }.\nonumber
\end{align}
This shows,
\begin{equation*}
    \sum\limits_{\Gamma_{N,M}(\overline{n})} \frac{1}{\langle n_1\dots+n_{k+\l}\rangle^{2\eps}}  \prod\limits_{j=1}^{k+\l}\zeta(n_j) \lesssim M^{-\frac{\eps}{2}}.
\end{equation*}
Using the Propositions \ref{PROP: Sobolev} and \ref{PROP: Weiner chaos} in a similar way to \eqref{EQU: Sobolev + Weiner Chaos} we have,
\begin{align}\label{EQU: Sobolev + Weiner third time}
\E [ \norm{:\!\Psi_N(\cdot,t)^k\overline{\Psi_N(\cdot,t)}^\l\!\!:-:\!\Psi_M(\cdot,t)^k\overline{\Psi_M(\cdot,t)}^\l\!\!:}_{W^{-\eps,\infty}}^p ] &\lesssim_{p,\eps} M^{-\frac{\eps}{2}}.
\end{align}
We now show a time difference estimate for $:\!\Psi^k_N\overline{\Psi}^\l_N\!\!:$. This will show that $:\!\Psi^k_N\overline{\Psi}^\l_N\!\!:$ is almost surely continuous in time and hence, combined with the previous part of this proof, $:\!\Psi^k_N\overline{\Psi}^\l_N\!\!:$ is Cauchy in $L^p(\Omega; C([0,T];W^{-\eps,\infty}))$. 

We define the time deference operator
\begin{equation*}
    \delta_h:\!\Psi_N^k\overline{\Psi_N}^\l(x,t)\!: \hspace{3pt}\stackrel{\text{def}}{=}\hspace{3pt} :\!\Psi_N^k\overline{\Psi_N}^\l(x,t+h)\!: -:\!\Psi_N^k\overline{\Psi_N}^\l(x,t)\!:
\end{equation*}
 for $|h|<1$. In the following we will assume $h>0$ for simplicity. Simple modifications are needed for the $h<0$ case. Expanding and then using Proposition \ref{orthoglag} we have,
\begin{align} 
    \E\bigg[\big(\delta_h:&\Psi_N^k\overline{\Psi_N}^\l(x,t)\!:\big)\overline{\big(\delta_h:\!\Psi_N^k\overline{\Psi_N}^\l(y,t)\!:\big)}\bigg]\label{EQU: time dif ortholag}\\
    &= \E\bigg[:\!\Psi_N^k\overline{\Psi_N}^\l(x,t+h)\!: \overline{:\!\Psi_N^k\overline{\Psi_N}^\l(y,t+h)\!:}\bigg]\nonumber\\
    &\hspace{10pt}-\E\bigg[:\!\Psi_N^k\overline{\Psi_N}^\l(x,t)\!: \overline{:\!\Psi_N^k\overline{\Psi_N}^\l(y,t+h)\!:}\bigg]\nonumber\\
    &\hspace{10pt}+\E\bigg[:\!\Psi_N^k\overline{\Psi_N}^\l(x,t)\!: \overline{:\!\Psi_N^k\overline{\Psi_N}^\l(y,t)\!:}\bigg]\nonumber\\
    &\hspace{10pt}-\E\bigg[:\!\Psi_N^k\overline{\Psi_N}^\l(x,t+h)\!: \overline{:\!\Psi_N^k\overline{\Psi_N}^\l(y,t)\!:}\bigg]\nonumber\\
    %%%
    & = C_{k,\l}\E\big[\Psi_N(x,t+h)\overline{\Psi_N(y,t+h)}\big]^{k}\overline{\E\big[\Psi_N(x,t+h)\overline{\Psi_N(y,t+h)}\big]}^{\l}\nonumber\\[5pt]
    & \hspace{10pt}-C_{k,\l}\E\big[\Psi_N(x,t)\overline{\Psi_N(y,t+h)}\big]^{k}\overline{\E\big[\Psi_N(x,t)\overline{\Psi_N(y,t+h)}\big]}^{\l}\nonumber\\[5pt]
    &\hspace{10pt}+ C_{k,\l} \E\big[\Psi_N(x,t)\overline{\Psi_N(y,t)}\big]^{k}\overline{\E\big[\Psi_N(x,t)\overline{\Psi_N(y,t)}\big]}^{\l}\nonumber\\[5pt]
    & \hspace{10pt}-C_{k,\l}\E\big[\Psi_N(x,t+h)\overline{\Psi_N(y,t)}\big]^{k}\overline{\E\big[\Psi_N(x,t+h)\overline{\Psi_N(y,t)}\big]}^{\l}\nonumber\\[5pt]
    & = (\textnormal{I}) + (\textnormal{II}) \nonumber
\end{align}
where in (I), we group the first and second terms on the right hand side of \eqref{EQU: time dif ortholag} and in (II) we group the third and fourth terms on the right hand side of \eqref{EQU: time dif ortholag}. Using the purely algebraic formula
\begin{equation}\label{EQU: poly dif}
    a^k\overline{a}^\l -b^k\overline{b}^\l = (a-b)\overline{a}^\l\sum\limits_{i=0}^{k-1}b^ia^{k-1-i}+\overline{(a-b)}b^k\sum\limits_{i=0}^{\l-1}\overline{b}^i\overline{a}^{\l-1-i}
\end{equation}
we can write (I) as,
\begin{align*}
    (\textnormal{I})&= C_{k,\l}\E\big[\delta_h\Psi_N(x,t)\overline{\Psi_N(y,t+h)}\big]\overline{\E\big[\Psi_N(x,t+h)\overline{\Psi_N(y,t+h)}\big]}^{\l}\nonumber\\[5pt]
    &\hspace{25pt}\times\sum\limits_{i=0}^{k-1}\E\big[\Psi_N(x,t)\overline{\Psi_N(y,t+h)}\big]^i\E\big[\Psi_N(x,t+h)\overline{\Psi_N(y,t+h)}\big]^{k-1-i}\nonumber\\[5pt]
    &\hspace{25pt} - C_{k,\l}\overline{\E\big[\delta_h\Psi_N(x,t)\overline{\Psi_N(y,t)}\big]}\E\big[\Psi_N(x,t)\overline{\Psi_N(y,t)}\big]^{\l} \label{EQU: time def poly dif est}\\[5pt]
    &\hspace{25pt}\times\sum\limits_{i=0}^{\l-1}\overline{\E\big[\Psi_N(x,t)\overline{\Psi_N(y,t)}\big]}^i\overline{\E\big[\Psi_N(x,t+h)\overline{\Psi_N(y,t)}\big]}^{\l-1-i}\nonumber \\[5pt]
    & = (\textnormal{Ia})+ (\textnormal{Ib}).\nonumber
\end{align*}
Using equation \eqref{EQU: poly dif} again we have a similar decomposition for (II),
\begin{equation*}
    (\textnormal{II}) = (\textnormal{IIa}) + (\textnormal{IIb}).
\end{equation*}
From \eqref{EQU: Simple Expect xy} we have,
\begin{align*}
    \E\big[\delta_h\Psi_N(x,t)\overline{\Psi_N(y,t+h)}\big] = \sum\limits_{|n|\leq N}e_n(x-y)\big(\zeta(n) - \zeta(n,t,t+h)\big)
\end{align*}
A similar equality holds for $\E\big[\delta_h\Psi_N(x,t)\overline{\Psi_N(y,t)}\big]$. Using the mean value theorem,
\begin{equation*}
    \big|\zeta(n) - \zeta(n,t,t+h)\big|\lesssim \min\big( |h|, \langle n\rangle^{-2}\big).
\end{equation*}
Hence by interpolation with \eqref{EQU: Zeta bound},
\begin{equation}\label{EQU: MVT app to zeta}
    \big|\zeta(n) - \zeta(n,t+t+h)\big|\lesssim |h|^{\alpha}\langle n\rangle^{2-2\alpha}.
\end{equation}
Taking the $\langle \nabla_x\rangle^{-\eps}$ and $\langle \nabla_y\rangle^{-\eps}$ Bessel potentials of \eqref{EQU: time dif ortholag}, setting $x=y$ and then using the estimates $\eqref{EQU: Zeta bound}$ and \eqref{EQU: MVT app to zeta} we have,
\begin{align*}
    \E\bigg[\big|\delta_h\big(\langle \nabla\rangle^{-\eps}:\!\Psi_N^k\overline{\Psi_N}^\l(\cdot,t)\!:\big)\big|^2\bigg]\lesssim |h|^\alpha \sum\limits_{|n_1|,\dots,|n_{k+\l}|\leq N}\frac{1}{\langle n_1+\cdots+n_{k+\l}\rangle^{2\eps}} \frac{1}{\langle n_1\rangle^{2-2\alpha}}\prod\limits_{j=2}^{k+\l}\frac{1}{\langle n_j\rangle^2}.
\end{align*}
If $\alpha<\eps$ the summation in the above equation can be summed using a method similar to \eqref{EQU: Est of sum 2}. Using Propositions \ref{PROP: Sobolev} and \ref{PROP: Weiner chaos},
\begin{align}\label{EQU: time dif fixed N}
    \E\bigg[\norm{\delta_h:\!\Psi_N^k\overline{\Psi_N}^\l(x,t)\!:}_{W^{-\eps,\infty}}^p\bigg]\lesssim |h|^{\alpha p}.
\end{align}
Choosing $p$ large enough so that $\alpha p>1$, the Kolmogorov continuity criterion, see \cite[Propostion 8.2]{Bass}, implies that $:\!\Psi_N^k\overline{\Psi_N}^\l\!: \in C([0,T];W^{-\eps,\infty})$ almost surely. For the convergence of $:\!\Psi_N^k\overline{\Psi_N}^\l\!:$ in $L^p(\Omega; C([0,T];W^{-\eps,\infty}))$, in a manner similar to \eqref{EQU: time dif fixed N} and \eqref{EQU: Sobolev + Weiner third time}, for $N\geq M\geq 1$ we can show
\begin{align*}\label{EQU: time dif N dif}
    \E\bigg[\norm{\delta_h\big(:\!\Psi_N^k\overline{\Psi_N}^\l(\cdot,t)\!:-:\!\Psi_M^k\overline{\Psi_M}^\l(\cdot,t)\!:\big)}_{W^{-\eps,\infty}}^p\bigg]\lesssim |h|^{\alpha p}M^{-\eps/2}.
\end{align*}
Choosing $p$ large enough so that $\alpha p>1$, from the Kolmogorov continuity criterion we have that $:\!\Psi_N^k\overline{\Psi_N}^\l\!:$ is a Cauchy sequence in $L^p(\Omega; C([0,T];W^{-\eps,\infty}))$ and so denoting it's limit by $:\!\Psi^k\overline{\Psi}^\l\!:$ we have that $:\!\Psi^k\overline{\Psi}^\l\!: \in C([0,T];W^{-\eps,\infty}))$ almost surely.
\end{proof}

\begin{remark}
\textnormal{The above argument can be easily adapted to show the paths of $:\!\!\Psi^k\overline{\Psi}^\l\!:$ are in $C^{\alpha}([0,T];C^{-\eps})$ almost surely for $\alpha<\eps$. See for example \cite{GKO}.}
\end{remark}

\begin{remark}\label{REM: Smooth truncation sto conv}
\textnormal{Similar calculations would show that Proposition \ref{PROP:sconv} also holds for a sharp frequency truncation of $\Psi$ instead of a smooth frequency truncation. In particular, if $\Plow$ is the smooth frequency projector defined by 
\begin{equation*}
    \Plow\left(\sum\limits_{n\in\Z^2} f_n e^{in\cdot x} \right)=\sum\limits_{|n|\leq N}f_ne^{in\cdot x},
\end{equation*}
then,
\begin{align*}
    \E\big[\norm{\Plow \Psi(t)}_{C^{-\eps}}^p\big]<C<\infty. 
\end{align*}}
\end{remark}

\section{Local well-posedness of the renormalized SCGL}\label{SEC:3}

\subsection{Statement of results}
In this section, we present the proof of Theorem \ref{LWP}. To do, this we reformulate Theorem \ref{LWP} in a way slightly more amenable to PDE techniques.  

For $\eps>0$ to be fixed later, we consider the space $\widehat{C}^{-\eps}_T$ of $(m+1)\times m$-tuples of functions in $C([0,T];C^{-\eps})$. That is $\vec{z}\in \widehat{C}^{-\eps}_T$ if 
\begin{equation*}
z_{i,j}\in C([0,T];C^{-\eps}) \mbox{ for all } 0\leq i\leq m \mbox{ and }   0\leq j\leq m-1.
\end{equation*}
We define a norm on $\widehat{C}^{-\eps}_T$ as follows:
\begin{equation*}
\norm{\vec{z}}_{\widehat{C}^{-\eps}_T}=\max\limits_{i,j}\norm{z_{i,j}}_{C([0,T];C^{-\eps})}.
\end{equation*}

Instead of studying \eqref{EQU: for v SCGL} directly, for $\vec{z}\in \widehat{C}^{-\eps}_T$ we study the equation 
\begin{align}
\begin{cases}
\dt v=(a_1+ia_2)[\Delta-1]v+F(v,\vec{z})\\
v|_{t = 0} = v_0
\end{cases}
\quad (x, t) \in \T^2\times \R_+
\label{SCGL11}
\end{align}
where
\begin{equation*}
F(v,\vec{z})=(c_1+ic_2)\sum\limits_{\substack{0\leq i\leq m\\0\leq j\leq m-1}}\binom{m}{i}\binom{m-1}{j}z_{m-i,m-1-j}v^i\overline{v}^j.
\end{equation*}
\noindent
The local well-posedness argument in this section will work for any choice of $\vec{z}\in \widehat{C}^{-\eps}_T$. Proposition \ref{PROP:sconv} shows that, $\{:\!\Psi^\ell\overline{\Psi}^k\!:\}_{i,j}\in \widehat{C}^{-\eps}_T$. Hence if we show that \eqref{SCGL11} is locally well-posed, Theorem \ref{LWP} will follow. The point of proving local well-posedness this way is that it draws a clear line between the probabilistic techniques used in the construction of the stochastic objects in Section \ref{SEC: Sto Conv} and the PDE techniques used in this section.

As usual, we interpret \eqref{SCGL11} in the mild sense. That is we say $v$ solves \eqref{SCGL11} on $[0,T]$ if for all $t\in[0,T]$,
\begin{equation}
v(t)=  S(t)v_0
 +\int_{0}^t  S(t-t') F(v,\vec{z}) dt'. 
\label{SCGL12}
\end{equation}

In the following we look for solutions in the Banach space $X^{s_1,s_2}_T$ defined through the norm
\begin{equation*}
    \norm{v}_{X^{s_1,s_2}_T} = \norm{v}_{\frac{s_2-s_1}{2},s_2,T} + \norm{v}_{L^\infty([0,T];C^{s_0})}
\end{equation*}
where
\begin{equation*}\label{EQU: time weighted norm}
    \norm{v}_{\alpha,\beta,T} = \sup\limits_{t\in [0,T]}t^\alpha\norm{v(t)}_{C^{\beta}}.
\end{equation*}

The goal in this section is to prove the following.

\begin{proposition}\label{solntoSCGL11}
Suppose $m\geq 2$, an integer, and $s_0<0$ are such that
\begin{equation}\label{indicecond}
-(2m-1)\frac{s_0}{2} <1.
\end{equation}
Then \eqref{SCGL11} is locally well-posed for initial data in $C^{s_0}$.
More precisely for $\eps>0$ small enough there exists $\theta>0$ such that for $R>1$, given $v_0\in C^{s_0}$ and $\vec{z}\in \widehat{C}^{-\eps}_T$  such that 
\begin{equation*}
\norm{\vec{z}}_{\widehat{C}^{-\eps}_T},\norm{v_0}_{C^{-s_0}}\leq R
\end{equation*}
there exists a unique solution, $v\in C([0,T];C^{s_0})\cap C((0,T];C^{2\eps})$ where $T\sim R^{-\theta}$. Moreover, if $v_0,u_0\in C^{-s_0}$ and $\vec{z},\vec{x}\in \widehat{C}^{-\eps}_T$ satisfy
\begin{equation*}
\norm{\vec{z}}_{\widehat{C}^{-\eps}_T},\norm{v_0}_{C^{s_0}},\norm{\vec{x}}_{\widehat{C}^{-\eps}_T},\norm{u_0}_{C^{s_0}} \leq R
\end{equation*}
then the respective solutions $v_1,v_2\in C((0,T];C^{-s_0})$ to \eqref{SCGL12} with initial data and forcing $v_0,\vec{z}$ and $u_0,\vec{x}$ satisfy 
\begin{equation*}
\norm{v_1-v_2}_{X^{s_0,2\eps}_T}\lesssim \norm{u_0-v_0}_{C^{s_0}}+\norm{\vec{z}-\vec{x}}_{\widehat{C}^{-\eps}_T}.
\end{equation*}
\end{proposition}

A similar local well-posedness result holds for $C^s$ initial data, but with time of existence depending on the $L^p$ norm of the initial data. See \cite{MW1} for a similar result for SQE.

\begin{proposition}\label{smoother ID}
Suppose $m\geq 2$ is an integer, $p>2m-1$ and $\eps>0$ is sufficiently small but fixed. Then for initial data in $C^{2\eps}$, there exists a unique solution to \eqref{SCGL11} in $C([0,T],C^{2\eps})$. Moreover, this solution depends continuously on $\vec{z}$ and $v_0$. More precisely there exists $\theta>0$ such that for $R>1$, given $v_0\in C^{2\eps}$ and $\vec{z}\in \widehat{C}^{-\eps}_T$  such that 
\begin{equation*}
\norm{\vec{z}}_{\widehat{C}^{-\eps}_T},\norm{v_0}_{L^p}\leq R
\end{equation*}
there exists a unique solution, $v\in C([0,T];C^{2\eps})$ where $T\sim R^{-\theta}$, to \eqref{SCGL12}. Moreover, if $v_0,u_0\in C^{2\eps}$ and $\vec{z},\vec{x}\in \widehat{C}^{-\eps}_T$ satisfy
\begin{equation*}
\norm{\vec{z}}_{\widehat{C}^{-\eps}_T},\norm{v_0}_{C^{2\eps}},\norm{\vec{x}}_{\widehat{C}^{-\eps}_T},\norm{u_0}_{C^{2\eps}} \leq R
\end{equation*}
then the respective solutions $v_1,v_2\in C([0,T];C^{s_0})$ to \eqref{SCGL12} with initial data and forcing $v_0,\vec{z}$ and $u_0,\vec{x}$ satisfy 
\begin{equation*}
\norm{v_1-v_2}_{C([0,T]; C^{2\eps})}\lesssim_{R} \norm{u_0-v_0}_{C^{2\eps}}+\norm{\vec{z}-\vec{x}}_{\widehat{C}^{-\eps}_T}.
\end{equation*}
\end{proposition}
In light of the instantaneous smoothing from regularity $s_0$ to $2\eps$ in Proposition \ref{solntoSCGL11}, this proposition will allow us to prove global well-posedness by demonstrating an a priori estimate on the growth of the $L^p$ norm of solutions to \eqref{SCGL11}.

\subsection{Proof of local well-posedness results}

Before we prove the above propositions we first state and prove a useful elementary lemma.

\begin{lemma}\label{intest}
Suppose $\alpha,\beta\in \R$ satisfy $\alpha<1$ and $\beta<1$. Then,
\begin{equation*}
\int_0^t (t-s)^{-\alpha}s^{-\beta}\,ds\sim t^{1-\alpha-\beta}.
\end{equation*}
\end{lemma}
\begin{proof}
We split the integral into two parts and estimate each piece separately:
\begin{align*}
\begin{split}
\int_0^t(t-s)^{-\alpha}s^{-\beta}\,ds&=\int_{0}^{t/2}(t-s)^{-\alpha}s^{-\beta}\,ds+\int_{t/2}^t(t-s)^{-\alpha}s^{-\beta}\,ds\\
&\sim t^{-\alpha}\int_0^{t/2}s^{-\beta}\,ds + t^{-\beta}\int_{t/2}^t(t-s)^{-\alpha}\,ds\\
&\sim t^{1-\alpha-\beta}
\end{split}
\end{align*}
where in the last line we simply evaluated the two integrals.
\end{proof}

The following local well-posedness proof, using the Da Prato-Debussche trick, the linear heat smoothing estimate \eqref{PROP: Heat smooth} and the product estimate \eqref{applicparaprod} is standard, see for example, \cite{DPD,MW1,TW}. For completeness we go through the argument here.

\begin{proof}[Proof of Proposition~\ref{solntoSCGL11}]
We will first show the existence of the solution in $X^{s_0,2\eps}_T$. 

Suppose $\norm{\vec{z}}_{\widehat{C}^{-\eps}_T},\norm{v_0}_{C^{s_0}}\leq R$. For $R_0>R$ yet to be chosen, let $B_{R_0}$ be the ball of radius $R_0$ and center $0$ in $X$. We aim to show the map
\begin{equation*}
\Gamma v(t)=S(t)v_{0}+\int_{0}^{t}S(t-t')F(v,\vec{z})(t')\,dt'
\end{equation*}
is a contraction mapping on $B_{R_0}$. The linear heat smoothing estimate, Proposition \ref{PROP: Heat smooth}, gives,
\begin{align*}
\norm{S(t)v_0}_{C^{s_0}}\lesssim t^{\frac{s_0-2\eps}{2}}\norm{v_0}_{C^{2\eps}}.
\end{align*}
From Proposition \ref{applicparaprod} we have,
\begin{align}\label{Fvz est}
\norm{F(v,\vec{z})(t')}_{C^{-\eps}}\lesssim \sum\limits_{\substack{0\leq i\leq m\\0\leq j\leq m-1}}\norm{z_{i,j}}_{C^{-\eps}}\norm{v}^{i+j}_{C^{2\eps}}\lesssim R_0^{2m-1}(t')^{(2m-1)\frac{s_0-2\eps}{2}}
\end{align}
where we used the fact that ${t'}^\frac{2\eps-s_0}{2}\norm{v(t')}_{C^{2\eps}}\leq R_0 $ for $v\in B_{R_0}$ and the fact that $\norm{z_{i,j}}_{C^{-\eps}}\leq\norm{\vec{z}}_{\widehat{C}^{-\eps}}\leq R\leq R_0$.
We then have,
\begin{align}\label{EQU: Est in LWP 1}
\begin{split}
t^{\frac{2\eps-s_0}{2}}\norm{\Gamma v}_{C^{2\eps}}&\lesssim \norm{v_0}_{C^{s_0}}+t^\frac{2\eps-s_0}{2}\int_{0}^{t}(t-t')^{-\frac{3\eps}{2}}\norm{F(v,\vec{z})(t')}_{C^{-\eps}}\,dt'\\
&\lesssim \norm{v_0}_{C^{s_0}}+R_0^{2m-1}t^\frac{2\eps-s_0}{2}\int_{0}^{t}(t-t')^{-\frac{3\eps}{2}}(t')^{(2m-1)\frac{s_0-2\eps}{2}}\,dt'\\
&\lesssim \norm{v_0}_{C^{s_0}}+R^{2m-1}_0t^{1-\frac{3\eps}{2}+(2m-2)\left(\frac{s_0-2\eps}{2}\right)}
\end{split}
\end{align}
where in the final inequality we used Lemma \ref{intest} and the condition \eqref{indicecond}. 
Taking a supremum we have
\begin{equation}\label{EQU: Time weight gamma est}
\norm{\Gamma v}_{\frac{2\eps-s_0}{2},2\eps,T}\leq C\norm{v_0}_{C^{s_0}}+CR^{2m-1}_0T^{1-\frac{3\eps}{2}+(2m-2)\left(\frac{s_0-2\eps}{2}\right)}\nonumber
\end{equation}
where $C$ is the implicit constant \eqref{EQU: Est in LWP 1}.

Similarly, choosing $\eps$ small enough so that $s_0<-\eps$ and using \eqref{Fvz est},
\begin{align}
 \norm{\Gamma v}_{C^{s_0}}&\lesssim \norm{v_0}_{C^{s_0}}+\int_0^t\norm{F(v,\vec{z})(t')}_{C^{-\eps}}\,dt'\nonumber\\
 &\lesssim \norm{v_0}_{C^{s_0}} + R_0^{2m-1}t^{1+(2m-1)\frac{s_0-2\eps}{2}}.\label{EQU: Gamma v bound}
\end{align}
Taking a supremum we have,
\begin{equation}\label{EQU: no time weight gamma est}
   \norm{\Gamma v}_{L^\infty([0,T];C^{s_0})} \leq C\norm{v_0}_{C^{s_0}} + CR_0^{2m-1}T^{1+(2m-1)\frac{s_0-2\eps}{2}}.
\end{equation}
Adding \eqref{EQU: no time weight gamma est} and \eqref{EQU: Time weight gamma est} and choosing $\eps$ small so that $\frac{s_0-2\eps}{2}<-\frac{3\eps}{2}$,
\begin{equation*}
    \norm{\Gamma v}_{X^{s_0,2\eps}_T}\leq C\norm{v_0}_{C^{s_0}}+CR_0^{2m-1}T^{1+(2m-1)\frac{s_0-2\eps}{2}}.
\end{equation*}
By the condition \eqref{indicecond} we can choose $\eps>0$ small enough so that
\begin{equation*}
\theta:= 1+(2m-1)\left(\frac{s_0-2\eps}{2}\right)>0
\end{equation*}
and so the power of $T$ is positive. Hence choosing $R_0=2CR$ and $T$ satisfying 
\begin{equation*}
  CR^{2m}T^{1-\frac{3\eps}{2}+(2m-2)\frac{s_0-2\eps}{2}}\leq \frac{R}{2}  
\end{equation*}
we find that $\Gamma$ maps $B_{R_0}$ to $B_{R_0}$. Now we verify the contraction property. It follows form Proposition \ref{PROP: algebra prop} and \eqref{EQU: poly dif} that for $v_1,v_2\in B_{R_0}$,
\begin{equation}
\norm{v_1^i\overline{v_1}^j-v_2^i\overline{v_2}^j}_{C^{2\eps}}\lesssim (t')^{(i+j-1)\frac{s_0-2\eps}{2}}R^{i+j-1}\norm{v_1-v_2}_{C^{2\eps}}.\nonumber
\end{equation}
Hence we have the difference estimate 
\begin{equation}\label{difest}
    \norm{F(v_1,\vec{z})(t')-F(v_2,\vec{z})(t')}_{C^{-\eps}}\lesssim (t')^{(2m-1)\frac{s_0-2\eps}{2}}R^{2m-1}\norm{v_1-v_2}_{C^{2\eps}}.
\end{equation}
Using this estimate and estimates similar to those in \eqref{EQU: Est in LWP 1} and \eqref{EQU: Gamma v bound} we can show that $\Gamma: B_{R_0}\rightarrow B_{R_0}$ is a contraction mapping. By the contraction mapping theorem it follows that $\Gamma$ has a unique fixed point and hence, \eqref{SCGL11} has a solution in $X^{s_0,2\eps}_T$.

Using Gr\"onwall and standard PDE techniques the uniqueness of the solution on $B_{R_0}$ can be extended to all of $X^{s_0,2\eps}_T$. Using \eqref{Prop: Time cont} and standard PDE techniques it can be shown that the solution we constructed above is in fact in $C((0,T];C^{2\eps})\cap C([0,T];C^{s_0})$. Further using standard PDE techniques it can be shown that the solution depends continuously on the noise and initial data. 

The proofs of these three statements are quite standard. We will just prove the continuous dependence. Let 
\begin{equation*}
v_1(t)=S(t)v_0+\int_0^tS(t-t')F(v_1,\vec{z})(t')\,dt'
\end{equation*}
and 
\begin{equation*}
v_2(t)=S(t)u_0+\int_0^tS(t-t')F(v_2,\vec{x})(t')\,dt'
\end{equation*}
be the solutions on $[0,T]$ for $T\sim R^{-\theta}$ constructed by the above contraction mapping argument. Adding and subtracting $F(v_1,\vec{x})$ we have
\begin{align*}
F(v_1,&\vec{z})-F(v_2,\vec{x})=\\
&+(c_1+ic_2)\sum\limits_{\substack{0\leq i\leq m\\0\leq j\leq m-1}}\binom{m}{i}\binom{m-1}{j}(z_{m-i,m-1-j}-x_{m-i,m-1-j})v_1^i\overline{v_1}^j\\
&+(c_1+ic_2)\sum\limits_{\substack{0\leq i\leq m\\0\leq j\leq m-1}}\binom{m}{i}\binom{m-1}{j}x_{m-i,m-1-j}(v_1^i\overline{v_1}^j-v_2^i\overline{v_2}^j)\\
&= \textnormal{(I)} + \textnormal{(II)}
\end{align*}
Using estimates similar to \eqref{Fvz est} to estimate \textnormal{(I)} and  \eqref{difest} to estimate \textnormal{(II)},
\begin{align*}
\norm{F(v_1,\vec{z})(t')-F(v_2,\vec{x})(t')}_{C^{-\eps}}&\lesssim R^{2m-2}_0(t')^{(2m-2)\frac{s_0-2\eps}{2}}\norm{v_1-v_2}_{\frac{2\eps-s_0}{2},2\eps,T}\\
&+R^{2m-1}_0(t')^{(2m-1)\frac{s_0-2\eps}{2}}\norm{\vec{z}-\vec{x}}_{\widehat{C}^{-\eps}_T}.
\end{align*}
So,
\begin{align*}
\norm{v_1-v_2}_{\frac{2\eps-s_0}{2},2\eps,T}  &\lesssim \norm{v_0-u_0}_{C^{s_0}}+T^{\frac{2\eps-s_0}{2}}\int_0^T(t')^{-\frac{3\eps}{2}}\norm{F(v_1,\vec{z})-F(v_2,\vec{x})}_{C^{-\eps}}\,dt'\\
&\lesssim \norm{v_0-u_0}_{C^{s_0}}\\
&\hspace{0.5cm}+T^{1-\frac{3\eps}{2}-(2m-3)\frac{2\eps-s_0}{2}}R_0^{2m-1}\left(\norm{\vec{z}-\vec{x}}_{\widehat{C}^{-\eps}_T} +\norm{v_1-v_2}_{X^{s_0,2\eps}_T} \right).
\end{align*}
Similarly we have
\begin{equation*}
   \norm{v_1-v_2}_{L^\infty([0,T];C^{s_0})}\lesssim \norm{v_0-u_0}_{C^{s_0}}+ T^{1-(2m-2)\frac{2\eps-s_0}{2}}R_0^{2m-1}\big(\norm{\vec{z}-\vec{x}}_{\widehat{C}^{-\eps}_T} +\norm{v_1-v_2}_{X^{s_0,2\eps}_T} \big). 
\end{equation*}
Adding the above estimates gives,
\begin{equation*}
    \norm{v_1-v_2}_{X^{s_0,2\eps}_T}\leq C\norm{v_0-u_0}_{C^{s_0}}+CT^{1-(2m-2)\frac{2\eps-s_0}{2}}R_0^{2m-1}\big(\norm{\vec{z}-\vec{x}}_{\widehat{C}^{-\eps}_T} +\norm{v_1-v_2}_{X^{s_0,2\eps}_T} \big).
\end{equation*}

Choosing $T$ small enough, we can bring $\frac{1}{2}\norm{v_1-v_2}_{X^{s_0,2\eps}_T}$ to the left hand side of the above inequality giving,
\begin{equation*}
\norm{v_1-v_2}_{X^{s_0,2\eps}_T}\lesssim \norm{v_0-u_0}_{C^{s_0}}+ \norm{\vec{z}-\vec{x}}_{\widehat{C}^{-\eps}_T}.
\end{equation*}
\end{proof}

We now outline the proof of Proposition \ref{smoother ID}. For more details we refer the reader to \cite{MW1} where a similar result is proven for the two dimensional SQE.

\begin{proof}[Proof of Proposition~\ref{smoother ID}]
Suppose $\norm{\vec{z}}_{\widehat{C}^{-\eps}},\norm{v_0}_{L^p}\leq R$. We let $B$ denote the Banach space defined through the norm $\norm{\cdot}_{\eps+\frac{1}{p},2\eps,T}$. Following \cite[Theorem 6.2]{MW1} we will first show that there exists a solution in $B$. Then we will show that the solution constructed is in fact in $C([0,T];C^{2\eps})$. For $R_0>R$ yet to be chosen let $B_{R_0}$ be the ball of radius $R_0$ and center $0$ measured in the norm $\norm{\cdot}_{\eps+\frac{1}{p},2\eps,T}$. From the mild formulation we have,
\begin{equation*}
\norm{\Gamma v(t)}_{C^{2\eps}}\leq \norm{S(t)v_0}_{C^{2\eps}}+\int_0^t(t-t')^{-\frac{3\eps}{2}}\norm{F(v,\vec{z})(t')}_{C^{-\eps}}\,dt'.
\end{equation*}
Using Proposition \ref{besovembed} and Proposition \ref{PROP: Heat smooth},
\begin{equation*}
    \norm{S(t)v_0}_{C^{2\eps}}\lesssim \norm{S(t)v_0}_{B^{2\eps+\frac{2}{p}}_{p,\infty}}\lesssim t^{-\eps-\frac{1}{p}}\norm{v_0}_{L^p}.
\end{equation*}
Using the above estimate and an estimate similar to \eqref{Fvz est} in the proof of Proposition \ref{solntoSCGL11} we have
\begin{equation*}
t^{\eps+\frac{1}{p}}\norm{\Gamma v(t)}_{C^{2\eps}}\lesssim \norm{v_0}_{L^p}+t^{\eps+\frac{1}{p}}\int_0^t(t-t')^{-\frac{3\eps}{2}}(t')^{-(2m-1)(\eps+\frac{1}{p})}R_0^{2m-1}\,dt'.
\end{equation*}
If $2m-1<p$ and $\eps>0$ is small enough then the integral in the above equation can be evaluated using Lemma \ref{intest} and taking a supremum,
\begin{equation*}
    \norm{\Gamma v}_{\eps+\frac{1}{p},2\eps,T}\leq C\norm{v_0}_{L^p}+CR_0^{2m-1}T^{1-\frac{3\eps}{2}-(2m-2)(\eps+\frac{1}{p})}.
\end{equation*}
Choosing $R_0=2CR$ and $T$ so that
\begin{equation*}
    CR_0^{2m-1}T^{1-\frac{3\eps}{2}-(2m-2)(\eps+\frac{1}{p})}\leq \frac{1}{2}R
\end{equation*}
it follows that $\Gamma$ maps $B_{R_0}$ to itself. Using arguments similar to those in the proof of Proposition \ref{solntoSCGL11} one can verify a difference estimate for $\Gamma$. Hence by the Contraction Mapping Theorem, $\Gamma$ has a fixed point.

Using \ref{Prop: Time cont} and arguments in \cite[Proposition 6.2]{MW1} one can show the solution constructed above is in fact in $C([0,T];C^{2\eps})$ and is unique in this space.

The proof of the continuous dependence on $v_0\in C^s$ and $\vec{z}$ is similar to the proof of continuous dependence in the proof of Proposition \ref{solntoSCGL11}.
\end{proof}

\section{Global well-posedness of the renormalized SCGL}\label{Sect: GWP}

In this section we place the additonal assumption that $c_1<0$. This means that the nonlinearity is defocusing with respect to the heat part of SCGL. 

In this section we will prove Theorem \ref{GWP}. To do this we will prove the following global well-posedness result for \eqref{SCGL11}. 

\begin{proposition}\label{GWPOFscgl11}
Let $m\geq 2$ be an integer and suppose $s_0>-\frac{2}{2m-1}$. Set $r=\big|\frac{a_2}{a_1}\big|$. Suppose 
\begin{equation*}
2m-1<2+2(r^2+2r\sqrt{1+r^2})
\end{equation*}
and suppose $\eps=\eps(m,r)>0$ is sufficiently small.
Then for any $T>0$, $v_0\in C^{s_0}$ and $\widehat{z}\in \widehat{C}^{-\eps}_T$ there exists a unique solution $v$ to \eqref{SCGL11} with  $v\in C((0,T];C^{2\eps})\cap C([0,T];C^{s_0})$.
\end{proposition}
\noindent
From Proposition \ref{PROP:sconv} $\{:\! \Psi^k\overline{\Psi^\ell}\!:\}_{k,\ell}\in \widehat{C}^{-\eps}_T$. Thus, if we can prove the above proposition, the Theorem \ref{GWP} will follow

To prove this, we will establish an a priori $L^p$ bound coming from a ``Testing against $v^{p-1}$'' identity. This is similar to the method in \cite{MW1,TW,MW2,Hos}. However, as in \cite{Hos} our situation is more delicate than the situation in \cite{MW1,TW,MW2}. Due to some extra terms appearing in our ``Testing against $v^{p-1}$'' identity, we are only be able to establish a suitable a priori $L^p$ bound for small $p$. However, for the $L^p$ norm to control the time of existence in Proposition \ref{smoother ID} we need $p>2m-1$. Hence we only get global well-posedness when these two ranges overlap.

We now state and prove the ``Testing against $v^{p-1}$'' identity previously alluded to.

\begin{proposition}\label{PROP: Lp inequality}
Let $T>0$ be fixed and $m\geq 2$ be an integer. Set $r=\big|\frac{a_2}{a_1}\big|$. Suppose 
\begin{equation*}
2m-1<2+2(r^2+2r\sqrt{1+r^2})
\end{equation*}
and suppose $\eps=\eps(m,r)>0$ is sufficiently small. Further, suppose $v_0 \in C^\infty_x$, $\widehat{z}\in \widehat{C}^\infty_T$ and $v\in C^\infty_tC^\infty_x$ solves \eqref{SCGL11}. Then, for $\eta >0$ small enough, $v$ satisfies the following inequality

\begin{align}\label{Lpineq}
\begin{split}
\frac{1}{p}\left(\norm{v(t)}_{L^p}^p-\norm{v(t_0)}_{L^p}^p\right)&+\int_{t_0}^t \norm{v^{p+2m-2}(t')}_{L^1}\,dt'+4\eta a_1\int_{t_0}^t\norm{v^{p-2}|\nabla v|^2(t')}_{L^1}\,dt'\\
&\leq\int_{t_0}^t|\langle F_0(v,\vec{z}),|v|^{p-2}v\rangle|(t')\,dt'
\end{split}
\end{align}
where
\begin{equation*}
F_0(v,\vec{z})=(c_1+ic_2)\sum\limits_{\substack{0\leq i\leq m\\0\leq j\leq m-1\\ (i,j)\neq (m,m-1)}}\binom{m}{i}\binom{m-1}{j}z_{m-i,m-1-j}v^i\overline{v}^j.
\end{equation*}
\end{proposition}

\begin{proof}
We will assume $a_2\geq 0$ as if $a_2<0$ we can take the conjugate of \eqref{SCGL12} so that $\overline{v}$ solves \eqref{SCGL12} with the sign of $a_2$ switched. As $v$ is sufficiently smooth we can compute, 
\begin{align}\label{partialLp}
\begin{split}
\frac{1}{p}\partial_t \norm{v(t)}_{L^p}^p&=\frac{1}{p}\partial_t \int_{\T^2}(v\overline{v})^{p/2}\,dx.\\
&=\frac{1}{2}\int_{\T^2}(v\overline{v})^{p/2-1}\left(v\partial_t\overline{v}+\overline{v}\partial_tv  \right)\,dx\\
&=\frac{1}{2}\int_{\T^2}(v\overline{v})^{p/2-1}\left((a_1-ia_2)v\Delta \overline{v}+(a_1+ia_2)\overline{v}\Delta v \right)\,dx\\
&\hphantom{=}+\int_{\T^2}(v\overline{v})^{p/2-1}\Re \left(\overline{v}F(v,\vec{z}\right)\,dx.
\end{split}
\end{align}
Integrating by parts and then applying the product rule gives,
\begin{align*}
\begin{split}
\frac{1}{2}\int_{\T^2}(v\overline{v})^{p/2-1}\left((a_1-ia_2)v\Delta \overline{v}\right)\,dx
&=-\frac{1}{2}(a_1-ia_2)\int_{\T^2}\nabla\left[(v\overline{v})^{p/2-1}v\right]\cdot\nabla\overline{v}\,dx\\
&=-\frac{p}{4}(a_1-ia_2)\int_{\T^2}|v|^{p-2}|\nabla v|^2\,dx\\
&\hphantom{=}-\frac{p-2}{4}(a_1-ia_2)\int_{\T^2}|v|^{p-4}v^2(\nabla\overline{v})^2\,dx.
\end{split}
\end{align*}
Here we are using the notation $v^2=v_1^2+v_2^2$ for $v\in \C^2$. Note that this is distinct from $|v|^2$. Using this expression we can write the first line in the third equality in \eqref{partialLp} as
\begin{align}\label{3rd line}
\begin{split}
\frac{1}{2}\int_{\T^2}(v\overline{v})^{p/2-1}\big((a_1-ia_2)v\Delta \overline{v}&+(a_1+ia_2)\overline{v}\Delta v \big)\,dx=\\
&-\frac{p}{2}a_1\int_{\T^2}|v|^{p-2}|\nabla v|^2\,dx\\
&-\frac{p-2}{4}a_1\int_{\T^2}|v|^{p-4}\left[v^2(\nabla \overline{v})^2+\overline{v}^2(\nabla v)^2  \right]\,dx\\
&-i\frac{p-2}{4}a_2\int_{\T^2}|v|^{p-4}\left[\overline{v}^2(\nabla v)^2-v^2(\nabla \overline{v})^2  \right]\,dx.
\end{split}
\end{align}
Making use of the identities
\begin{equation*}
v^2(\nabla \overline{v})^2+\overline{v}^2(\nabla v)^2=(v\nabla\overline{v}-\overline{v}\nabla v)^2+2|v|^2|\nabla v|^2,
\end{equation*}
\begin{equation*}
\nabla |v|^2=\overline{v}\nabla v+v\nabla\overline{v}
\end{equation*}
and
\begin{equation*}
4|v|^2|\nabla v|^2=(\nabla |v|^2)^2-(v\nabla\overline{v}-\overline{v}\nabla v)^2,
\end{equation*}
for $\eta>0$ we can write \eqref{3rd line} as
\begin{align*}
\begin{split}
\frac{1}{2}\int_{\T^2}(v\overline{v})^{p/2-1}\big(&(a_1-ia_2)v\Delta \overline{v}+(a_1+ia_2)\overline{v}\Delta v \big)\,dx=\\[5pt]
&\hphantom{=}-(p-1)a_1\int_{\T^2}|v|^{p-2}|\nabla v|^2\,dx
\frac{p-2}{4}a_1\int_{\T^2}|v|^{p-4}(v\nabla\overline{v}-\overline{v}\nabla v)^2\,dx\\[5pt]
&\hphantom{=}-\hphantom{=}i\frac{p-2}{4}a_2\int_{\T^2}|v|^{p-4}\nabla|v|^2\left( \overline{v}\nabla v-v\nabla\overline{v} \right)\,dx\\[5pt]
&=-4\eta a_1\int_{\T^2}|v|^{p-2}|\nabla v|^2\,dx
-\big(\frac{p-1}{4}-\eta\big)a_1\int_{\T^2}|v|^{p-4}(\nabla |v|^2)^2\,dx\\[5pt]
&\hphantom{=}+\big(\frac{1}{4}-\eta\big)a_1\int_{\T^2}|v|^{p-4}(v\nabla\overline{v}-\overline{v}\nabla v)^2\,dx\\[5pt]
&\hphantom{=}-\frac{p-2}{4}a_2\int_{\T^2}|v|^{p-4}\nabla|v|^2 i \left( \overline{v}\nabla v-v\nabla\overline{v} \right)\,dx\\[5pt]
&=-4\eta a_1\int_{\T^2}|v|^{p-2}|\nabla v|^2\,dx-\int_{\T^2}|v|^{p-4}A_{p,\eta}(f,g)\,dx
\end{split}
\end{align*}
where 
\begin{equation*}
f=i\left( \overline{v}\nabla v-v\nabla\overline{v} \right),\quad g=\nabla |v|^2
\end{equation*}
and $A_{p,\eta}(f,g)$ is the quadratic form
\begin{equation*}
A_{p,\eta}(f,g)=\big(\frac{1}{4}-\eta\big)a_1f^2+\frac{p-2}{4}a_2fg+\big(\frac{p-1}{4}-\eta\big)a_1g^2. 
\end{equation*}
Note that both $f$ and $g$ are real valued and so the quadratic form $A_{p,\eta}$ takes real arguments. If 
\begin{equation*}
2<p<2+2r(r+\sqrt{1+r^2})
\end{equation*}
then for small enough $\eta=\eta(p,r)$ the matrix 
\begin{equation*}
a_2\begin{pmatrix}
(\frac{1}{4}-\eta)r & \frac{p-2}{8}\\
\frac{p-2}{8} & (\frac{p-1}{4}-\eta)r
\end{pmatrix}
\end{equation*}
is non-negative definite, has non-negative trace and non-negative determinant, and so $A_{p,\eta}(f,g)\geq 0$. In this case
\begin{equation}\label{Lpineq1sthalf}
\frac{1}{2}\int_{\T^2}(v\overline{v})^{p/2-1}\left((a_1-ia_2)v\Delta \overline{v}+(a_1+ia_2)\overline{v}\eta v \right)\,dx\leq -4\eta a_1\int_{\T^2}|v|^{p-2}|\Delta v|^2\,dx.
\end{equation}
Note that the left hand side of the above inequality is real valued and so the inequality makes sense. We now consider the term on the second line of \eqref{partialLp}. As
\begin{equation*}
F(v,\vec{z})=(c_1+ic_2)|v|^{2m-2}v+F_0(v,\vec{z})
\end{equation*}
we have
\begin{equation}\label{Lpineq2half}
\int_{\T^2}(v\overline{v})^{p/2-1}\Re \left(\overline{v}F(v,\vec{z}\right)\,dx=c_1\int_{\T^2}|v|^{p+2m-2}\,dx+\Re\int_{\T^2}|v|^{p-2}\overline{v}F_0(v,\vec{z}).
\end{equation}
Putting \eqref{partialLp}, \eqref{Lpineq1sthalf} and \eqref{Lpineq2half} together after integrating from $t_0$ to $t$ gives the desired $L^p$ inequality
\begin{align*}
\frac{1}{p}\left(\norm{v(t)}_{L^p}^p-\norm{v(t_0)}_{L^p}^p\right)&+\int_{t_0}^t \norm{v^{p+2m-2}(t')}_{L^1}\,dt'+4\eta a_1\int_{t_0}^t\norm{v^{p-2}|\nabla v|^2(t')}_{L^1}\,dt'\\
&\leq\int_{t_0}^t|\langle F_0(v,\vec{z}),|v|^{p-2}v\rangle|(t')\,dt'.
\end{align*}
\end{proof}

It is not immediately clear how this proposition helps prove global well-posedness. Proposition \ref{PROP: Lp inequality} only holds for smooth initial data, noise and solutions to \eqref{SCGL11}. Without knowing the time continuity properties of $v$ it is not even clear \eqref{Lpineq} even makes sense for rough solutions of \eqref{SCGL11}. In \cite{MW1} this problem was solved by proving a certain amount of time continuity of $v$ and then proving an a priori bound of type \eqref{Lpineq} for rough $v$. In this paper we take an alternative PDE approach which we outline here.

Consider the solution $v_N$ of \eqref{SCGL12} with truncated forcing and initial data. That is the equation,
\begin{align}\label{SCGL11trunc}
\begin{cases}
\dt v_N=[(a_1+a_2)\Delta-1]v_N+F(v_N,\vec{z}_N)\\
v|_{t = 0} = \Slow v_0
\end{cases}
\end{align}
where 
\begin{equation*}
\vec{z}_N=\{\Slow z_{i,j}\}_{i,j}.
\end{equation*}
It can be shown that $v_N\in C_t^\infty C_x^\infty$ and hence $v_N$ is sufficiently regular for the hypothesis of Proposition \ref{PROP: Lp inequality} to hold. We then prove an a priori $L^p$ bound on $v_N$ that is independent of $N$. Using the fact that $\vec{z}_N\rightarrow \vec{z}$ in $\vec{C}^{-\eps}_T$ and $\Slow v_0\rightarrow v_0$ in $C^{s_0}$ one can can use the continuous dependence of the solution on $v_0$ and $\vec{z}$, from Section \ref{SEC:3} to show that $v$ is also a global solution.

With this in mind, to prove Proposition \ref{GWPOFscgl11}, it suffices to prove the following bound.

\begin{proposition}\label{truncLpbound}
Suppose $2<p<2+2(r^2+2r\sqrt{1+r^2})$ and $\eps>0$ is sufficiently small. Let $T>0$ and $0<t_0<T$. Then there exists $C=C(m,p,\eps,\vec{z})>0$ such that if $v_N$ is a solution to $\eqref{SCGL11trunc}$ on $[0,T]$ then for all $t\in [t_0,T]$,
\begin{equation*}
\norm{v_N(t)}_{L^p}\leq \norm{v_N(t_0)}_{L^p}+Ct.
\end{equation*}
\end{proposition}

To prove this Proposition we use an almost identical proof to that in \cite{MW1,TW}. For completeness we present the details here.
\begin{proof}[Proof of Proposition ~\ref{truncLpbound}]
In the following we write $v$ instead of $v_N$ for simplicity. Set
\begin{equation*}
A_t=4\eta c_1\norm{v^{p-2}(t)|\nabla v(t)|^2}_{L^1} \quad \mbox{ and }\quad B_t=\norm{v^{p+2m-2}(t)}_{L^1}.
\end{equation*}
Recall
\begin{equation*}
F_0(v,\vec{z}_N)=(c_1+ic_2)\sum\limits_{\substack{0\leq i< m\\0\leq j\leq m-1\\ (i,j)\neq (m,m-1)}}\binom{m}{i}\binom{m-1}{j}\Slow z_{m-i,m-1-j}v^i\overline{v}^j.
\end{equation*}
From \eqref{Lpineq} it suffices to show,
\begin{equation*}
|\langle |v|^{p-2}\overline{v},F_0(v,\vec{z}_N  \rangle| \leq A_t+B_t+C
\end{equation*}
for some constant $C$. To do this it suffices to prove
\begin{equation}\label{whatwewant}
\lvert\langle  \Slow z_{m-i,m-1-j}v^i\overline{v}^j, |v|^{p-2}v \rangle \rvert\leq \delta(A_s+B_s)+C(\delta)
\end{equation}
for some small $\delta>0$, for each $i$ and $j$.

We will just prove \eqref{whatwewant} for the case $(i,j)=(m-1,m-1)$, the other cases are similar, and in fact slightly easier as the homogeneity in $v$ is lower. By Proposition \ref{besovdual}, the boundedness of $\Slow$, 
\begin{equation*}
    \norm{\Slow f}_{C^{-\eps}}\leq \norm{f}_{C^{-\eps}},
\end{equation*}
see \cite{BCD}, and the fact $\norm{z_{1,0}}_{L^\infty([0,\infty);C^{-\eps})}<\infty$ we have,
\begin{align*}
\lvert\langle  \Slow z_{1,0}v^{m-1}\overline{v}^{m-1}, |v|^{p-2}v \rangle \rvert&=|\langle \Slow z_{1,0}, |v|^{p-2}v \overline{v}^{m-1}v^{m-1} \rangle|\\[5pt]
&\lesssim \norm{|v|^{2m-4+p}v}_{B_{1,1}^{\eps}}\norm{\Slow z_{1,0}}_{C^{-\eps}}\\[5pt]
&\lesssim \norm{|v|^{2m-4+p}v}_{B_{1,1}^{\eps}}.
\end{align*}
Applying Proposition \ref{nablaest}, Cauchy-Schwarz and then Jensen's inequality gives,
\begin{align}\label{Eq: Cauchy+Jensen}
\begin{split}
\norm{|v|^{2m-4+p}v}_{B_{1,1}^{\eps}}&\lesssim \norm{v^{2m-3+p}}_{L^1}^{1-\eps}\norm{v^{2m-4+p}\nabla v}_{L^1}^{\eps}+\norm{v^{2m-3+p}}_{L^1}\\[5pt]
&\lesssim \norm{v^{p-2}\lvert \nabla v\rvert^2}_{L^1}^{\eps/2}\norm{v^{p+4m-6}}_{L^1}^{\eps/2}\norm{v^{2m-3+p}}_{L^1}^{1-\eps}+\norm{v^{2m-3+p}}_{L^1}\\[5pt]
&\lesssim A_t^{\eps/2}B_t^{\frac{2m-3+p}{2m-2+p}(1-\eps)  }\norm{v^{p+4m-6}}_{L^1}^{\eps/2}+B_t^{\frac{2m-3+p}{p+2m-2}}.
\end{split}
\end{align}

Note that, as $x\mapsto x^{\frac{p+4m-6}{p+2m-2}}$ is not concave for $m\geq 2$, we cannot use Jensen's inequality to control $\norm{v^{p+4m-6}}_{L^1}$ by a power of $B_t$. To get around this problem we use a trick in \cite{TW}. Using Proposition \ref{PROP: Sobolev} in the form
\begin{equation*}
\norm{f}_{L^q}\lesssim \norm{f}_{L^2}+\norm{\nabla f}_{L^2}
\end{equation*}
which holds for $q<\infty$. With $f=v^{p/2}$ this implies
\begin{equation*}
\norm{v^{pq/2}}_{L^1}^{1/2}\lesssim \norm{v^p}_{L^1}^{q/4}+\norm{v^{p-2}\lvert\nabla v\rvert^2}_{L^1}^{q/4}.
\end{equation*}
 In particular with $q=\frac{2(p+4m-6)}{p}$ we have,
\begin{equation}\label{Eq: Consequence of Sobolev}
\norm{v^{p+4m-6}}_{L^1}^{1/2}\lesssim  \norm{v^p}_{L^1}^{\frac{p+4m-6}{2p}}+\norm{v^{p-2}\lvert\nabla v\rvert^2}_{L^1}^{\frac{p+4m-6}{2p}}\lesssim B_t^{\frac{p+4m-6}{2(p+2m-2)}}+A_t^{\frac{p+4m-6}{2p}}
\end{equation}
where the second inequality follows from Jensen's inequality which is now applicable. Putting \eqref{Eq: Cauchy+Jensen} and \eqref{Eq: Consequence of Sobolev} together gives
\begin{align*}
\lvert\langle  \Slow z_{1,0}v^{m-1}\overline{v}^{m-1}, |v|^{p-2} \rangle \rvert &\lesssim   A_t^{\eps/2+\frac{p+4m-6}{2p}\eps}B_t^{\frac{2m-3+p}{2m-2+p}(1-\eps)}\\
&+A_t^{\eps/2}B_t^{\frac{(2m-3+p)}{2m-2+p}(1-\eps)+\frac{p+4m-6)}{2(p+2m-2)}\eps}+B_t^{\frac{2m-3+p}{p+2m-2}}.
\end{align*}
Choosing  $\eps$ small enough so that
\begin{equation*}
\eps/2+\frac{p+4m-6}{2p}\eps+\frac{2m-3+p}{2m-2+p}(1-\eps)< 1
\end{equation*}
and
\begin{equation*}
\eps/2+\frac{(2m-3+p)}{2m-2+p}(1-\eps)+\frac{p+4m-6}{2(p+2m-2)}\eps< 1
\end{equation*}
we can use Young's inequality to get,
\begin{equation*}
\lvert\langle  \Slow z_{1,0}v^{m-1}\overline{v}^{m-1}, |v|^{p-2}v \rangle \rvert\leq \delta(A_t+B_t)+C(\delta).
\end{equation*}
Here we are choosing a preliminary $\delta'=\frac{\delta}{C}$ to absorb the implicit constants in the preceding inequalities.
This completes the proof.
\end{proof}

\section{Invariant measure and almost sure global well-posedness}
In this section we use an invariant measure argument to prove almost sure global well-posedness. In particular, we are able to prove almost sure global well-posedness for arbitrarily small dissipatation/dispersion ratios going beyond the global well-posedness result using the energy estimate in Section \ref{Sect: GWP}.

\subsection{On the Gibbs measure}

In this subsection we breifly discuss, by referring to results already in the literature, how to make sense of the measure \eqref{EQU: inv meas}. For more details see \cite{OT1, Nelson}. These results will be important in the proof of Theorem \ref{Th: AS GWP}.

Consider the measure $\mu_a$ which is induced under the map
\begin{equation}\label{EQU: Random id}
\omega\in\Omega \mapsto u(x)=u(x;\omega)=\sum\limits_{n\in\Z^2}\frac{g_n(\omega)}{\sqrt{a(1+|n|^2)}}e^{i n\cdot x}.
\end{equation}
This measure is, for any $s<0$, a mean-zero Gaussian measure on $H^{s}(\T^2)$ with covariance operator $Q_s=(a\textnormal{Id}-a\Delta)^{-1+s}$. The measure $\mu_{a}$ is formally given by
\begin{equation*}
d\mu_a  ``="  Z^{-1}_{a}e^{-\frac{a}{2}\int |\nabla u|^2dx-\frac{a}{2}\int|u|^2dx}du .
\end{equation*}

In the following it will be useful to decompose $\mu^a$ in the form 
\begin{equation*}
    \mu_a=\mu_a^N\otimes\mu_a^{N,\perp}
\end{equation*}
where $\mu_a^N$ is the measure induced under the map
\begin{equation*}
\omega\in\Omega \mapsto u_N(x)=u_N(x;\omega)=\sum\limits_{|n|\leq N}\frac{g_n(\omega)}{\sqrt{a(1+|n|^2)}}e^{i n\cdot x}
\end{equation*}
and $\mu_a^{N,\perp}$ is the measure induced under the map
\begin{equation*}
\omega\in\Omega \mapsto u(x)=u(x;\omega)=\sum\limits_{|n|> N}\frac{g_n(\omega)}{\sqrt{a(1+|n|^2)}}e^{i n\cdot x}.
\end{equation*}
Note that if $u(x)$ is given by \eqref{EQU: Random id} with $a = \frac{a_1}{\gamma}$ then for all $t\in \R$, $\textnormal{Law}(\Slow u) = \textnormal{Law}(\Slow \Psi(t))$.

For $u_N$ in the support of $\mu^{N}_a$ we define the Wick order monomial\footnote{Here it makes sense to talk about $:\!|u_N|^{2m}\!:$ as a Wick ordered monomial because of the Gaussian structure in \eqref{EQU: Random id}. See \cite{Nelson} for more information on Wick ordering}  
\begin{equation*}
:\! |\Slow u_N|^{2m}\!:=(-1)^m m! L_m(|\Slow u_N|^2;\sigma_N).
\end{equation*}
Set 
\begin{equation*}
    G^N_{c,2m}(u)=e^{-\frac{c}{2m}\int_{\T^2} :|\Slow u_N|^{2m}:\,dx }.
\end{equation*}
The following convergence result can be proven using standard techniques in the constructive quantum field theory literature. See for example \cite{OT1}.
\begin{proposition}\label{PROP: G bound indep N}
Let $m\geq 2$ be an integer and $c>0$. Then $G^N_{c,2m}\in L^p(\mu^a)$ for any $1\leq p<\infty$, with a bound uniform in $N$. Further, the sequence $\{G^N_{c,2m}\}_{N\in \mathbb{N}}$ converges in $L^p(\mu_a)$ to some, non zero,  $G_{c,2m}\in L^p(\mu_a)$. 
\end{proposition}
The above result gives a rigorous construction of the measure $P_{a_1,c_1,\gamma,2m}$: we define
\begin{equation*}
    dP_{a_1,c_1,\gamma,2m}:=Z_{a_1,c_1,\gamma,2m}^{-1}G_{\frac{c_1}{\gamma},2m}d\mu_\frac{a_1}{\gamma}.
\end{equation*}
Here $Z_{a_1,c_1,\gamma,2m}$ is a normalization factor to make $P_{a_1,c_1,\gamma,2m}$ a probability measure,
\begin{equation*}
    Z_{a_1,c_1,\gamma,2m} = \int_{C^{-\eps}}G_{\frac{c_1}{\gamma},2m}d\mu_\frac{a_1}{\gamma}.
\end{equation*}
One can also show that the truncated measures 
\begin{equation}\label{EQU: trunc meas}
    dP^N_{a_1,c_1,\gamma,2m}:= {Z^N_{a_1,c_1,\gamma,2m}}^{-1}G^N_{\frac{c_1}{\gamma},2m}d\mu_\frac{a_1}{\gamma}
\end{equation}
converge weakly, as $N\rightarrow\infty$, to $dP_{a_1,c_1,\gamma,2m}$, see for example \cite{OT1}. Here again, ${Z^N_{a_1,c_1,\gamma,2m}}$ is a normalization factor to make $P_{a_1,c_1,\gamma,2m}$ a probability measure,
\begin{equation*}
    Z^N_{a_1,c_1,\gamma,2m} = \int_{C^{-\eps}}G_{\frac{c_1}{\gamma},2m}^Nd\mu_\frac{a_1}{\gamma}.
\end{equation*}
We note that, as $Z^N_{a_1,c_1,\gamma,2m}$ is positive for every $N$, and $Z_{a_1,c_1,\gamma,2m}\neq 0$, there exists $\delta$ such that
\begin{equation}\label{EQU: Z bounded below}
    Z^N_{a_1,c_1,\gamma,2m} >\delta\quad \textnormal{ for all } N.
\end{equation}

\subsection{On the truncated equation and measure}

We will prove Theorem \ref{Th: AS GWP} using the truncated measure defined in the previous section, \eqref{EQU: trunc meas}, and the following truncated version of \eqref{EQU: renormalised equ for u}:
\begin{align}\label{EQU: SMOOTH renormalised equ for u_N}
\begin{cases}
\dt u_N=(a_1+ia_2)[\Delta-1] u_N -\mathcal{N}(u_N)+\sqrt{2\gamma}\xi  \\
u_N|_{t = 0} = u_0
\end{cases}
\end{align}
where
\begin{equation*}
    \mathcal{N}(u_N) = (c_1+ic_2)(-1)^{m-1}(m-1)!\Slow\left[L^{(1)}_{m-1}(|\Slow u_N|^2;\sigma_N) \Slow u_N\right]
\end{equation*}

In a manner similar to the introduction, we can decompose a solution, $u_N$, to \eqref{EQU: SMOOTH renormalised equ for u_N} in the form $u_N= v_N+\Psi$ where $v_N$ solves:
\begin{align}\label{EQU: SMOOTH trunc v SCGL}
\begin{cases}
\dt v_N=(a_1+ia_2)[\Delta-1] v_N- (c_1+ic_2)\Slow\hspace{-15pt}\sum\limits_{\substack{0\leq i\leq m\\0\leq j\leq m-1}}\hspace{-10pt}\binom{m}{i}\binom{m-1}{j}(\Slow v_N)^i\overline{(\Slow v_N)^j}:\!\Psi_N^{m-i}\overline{\Psi_N^{m-j-1}}\!:  \\[-10pt]
v_N|_{t = 0} = u_0-\Psi(0)
\end{cases}
\end{align}
where recall $\Psi_N = \Slow \Psi$.

We have the following result on the well-posednes of \eqref{EQU: SMOOTH trunc v SCGL}.

\begin{proposition}\label{PROP: SMOOTH vN is GWP}
For $N\in \mathbb{N}$ and $s_0>-\frac{2}{2m-1}$, \eqref{EQU: SMOOTH trunc v SCGL} is pathwise globally well-posed with initial data measured in $C^{s_0}$.
\end{proposition}
\begin{proof}
Here we just sketch to proof. The equation \eqref{EQU: SMOOTH trunc v SCGL} can be decoupled into a low frequency equation and a high frequency equation. Global well-posedness of the high frequency part of the equation is immediate as it is linear. Local well-posedness of the low frequency part of the equation follows from the Cauchy-Lipschitz theorem for stochastic processes. This can be extended to global well-posedness by proving a basic $L^2$ energy estimate, similar to that in Proposition \ref{truncLpbound} and then the Burkholder-Davis-Gundy inequality. For more details about this procedure, see \cite[Proposition 5.1]{ORT}.
\end{proof}

Using the local well-posedness theory developed in Section 5 we can prove the following. 

\begin{proposition}\label{PROP: vN approximates v}
Let $u_0\in C^{s_0}$ and let $T=T(\omega)$ be the maximal time of existence of \eqref{EQU: for v SCGL}. Then the solution, $v_N$, to \eqref{EQU: SMOOTH trunc v SCGL} converges to $v$, the solution, to \eqref{EQU: for v SCGL} in $C([0,T]; C^{s_0})$ as $N\rightarrow\infty$ for any $T<T(\omega)$. 
\end{proposition}

Proposition \ref{PROP: SMOOTH vN is GWP} allows us to define the solution map associated to equation \eqref{EQU: SMOOTH renormalised equ for u_N}, $$\Phi_N(\cdot,\cdot,\cdot):\R\times C^{-\eps} \times\Omega\rightarrow C^{-\eps}.$$ That is $\Phi_N(t,f,\omega)$ is the solution to \eqref{EQU: SMOOTH renormalised equ for u_N} at time $t$ starting from initial data $f\in C^{-\eps}$. We then define $P_t^N:C_b(C^{-\eps})\rightarrow B_b(C^{-\eps})$, the transition semi-group associated to \eqref{EQU: SMOOTH renormalised equ for u_N}, as follows:
\begin{equation*}
    P_t^N\psi(f)=\E[ \psi(\Phi_N(t,f,\omega)) ].
\end{equation*}
Here $B_b(C^{-\eps})$ is the set of all Borel-bounded functions from $C^{-\eps}$ to $\C$ and $C_b(C^{-\eps})$ is the set of all continuous Borel-bounded functions from $C^{-\eps}$ to $\C$. 

We say a measure $\nu$ on $C^{-\eps}$ is invariant under $P_t^N$ if 
\begin{equation*}
    \int_{C^{-\eps}} P_t^N\psi\,d\nu=\int_{C^{-\eps}}\psi\,d\nu
\end{equation*}
for all $\psi\in C_b(C^{-\eps})$ and for all $t\geq 0$. See \cite{DaPrato} for equivalent characterizations of what it means for a measure to be invariant.

\begin{proposition}   
Assume $\frac{a_1}{a_2}=\frac{c_1}{c_2}$. Then the probability measure $dP^N_{a_1,c_1,\gamma,2m}$ is an invariant measure for $P^N_t$.
\end{proposition}  
\begin{proof}
In this proof we heavily use ideas from \cite[Proposition 4]{DDF}. We present details here for completeness.

Equation \eqref{EQU: SMOOTH trunc v SCGL} decouples into a finite dimensional system of SDEs, corresponding to frequencies $|n|\leq N$,
\begin{align}\label{EQU: low freq SMOOTH renormalised equ for u_N}
\begin{cases}
d u_N=(a_1+ia_2)[\Delta-1] u_Ndt -\mathcal{N}(u_N)dt+\sqrt{2\gamma}\Plow dW(t)  \\
u_N|_{t = 0} = \Plow u_0
\end{cases}
\end{align}
and an infinite dimensional system of linear SDEs corresponding to the frequencies $|n|> N$.
\begin{align}\label{EQU: high freq SMOOTH renormalised equ for u_N}
\begin{cases}
d u_N=(a_1+ia_2)[\Delta-1] u_N\,dt+\sqrt{2\gamma}P_{> N}dW(t)  \\
u_N|_{t = 0} = P_{> N} u_0.
\end{cases}
\end{align}
Hence it suffices to show that the low frequency component of $dP^N_{a_1,c_1,\gamma,2m}$,  ${Z^N_{a_1,c_1,\gamma,2m}}^{-1}G^N_{\frac{c_1}{\gamma},2m}d\mu_{\frac{a_1}{\gamma}}^N$ is an invariant measure for \eqref{EQU: low freq SMOOTH renormalised equ for u_N} and the high frequency component of $dP^N_{a_1,c_1,\gamma,2m}$, $d\mu_{\frac{a_1}{\gamma}}^{N,\perp}$ is an invariant measure for \eqref{EQU: high freq SMOOTH renormalised equ for u_N}. For the high frequency component, we simply note that the solution of \eqref{EQU: high freq SMOOTH renormalised equ for u_N} is  $P_{>N} \Psi(t)$. By looking at the random Fourier series that defines this process, it is easy to check that $P_{>N} \Psi(t)$ is a stationary process with law $d\mu_{\frac{a_1}{\gamma}}^{N,\perp}$. 

For the low frequency measure, we use a method in \cite{DDF}. We work in the finite dimensional space $E_N = \mbox{span}\{e^{i n \cdot x}\} $. We view an element of $E_N$ as a vector composed of it's real and imaginary parts, $y=y_1+iy_2=(y_1,y_2)$ and write \eqref{EQU: low freq SMOOTH renormalised equ for u_N} in the form
\begin{equation}\label{EQU: low freq real imag}
    d(\Re u_N, \Im u_N)= -\frac{a_2}{a_1}JDI(\Re u_N, \Im u_N)-DI(\Re u_N, \Im u_N)+\sqrt{2\gamma}\Plow d(\Re W, \Im W).
\end{equation}
Here, $J$ is an analogue of multiplication by $i$,
\begin{equation*}
    J(y) = J(y_1+iy_2)=J(y_1,y_2) = (-y_2,y_1)=-y_2+iy_1=iy
\end{equation*}
and for $y=y_1+iy_2\in E_N$, 
\begin{align*}
    I(y)=I(y_1+iy_2)=I(y_1,y_2)&=\frac{a_1}{2}\int_{\T^2}\big(|\nabla y_1|^2+|\nabla y_2|^2+|y_1|^2+|y_2|^2\big)\\
    &\hspace{10pt}+\frac{c_1}{2m}\int_{\T^2}(-1)^{m-1}(m-1)!L_{m-1}^{(1)}(|\Slow y_1|^2+|\Slow y_2|^2; \sigma_N).
\end{align*}
Here we are using the notation that if $K:E_N\rightarrow \R$, then $DK(x)$ is the Fr\'echet derivative of $K$ at $x\in E_N$ which we identify with an element of $E_N$. 

The generator of \eqref{EQU: low freq real imag} (see \cite{DPD} for more information of generators of finite dimensional SDEs), $L_N$ is
\begin{align*}
    L_Nf(y) &= \gamma \textnormal{Tr}D^2f_1(y)+\gamma \textnormal{Tr}D^2f_2(y)-\frac{a_2}{a_1}\langle Df(y),JDI(y)\rangle_{E_N}-\langle Df(y),DI(y)\rangle_{E_N}
\end{align*}
where $f=(f_1,f_2)$ and $f_1,f_2: E_N\rightarrow \R$. Here for a twice differentiable function $g:E_N\rightarrow \R$,
\begin{equation*}
   \textnormal{Tr}D^2g(y) =\sum\limits_{|n|\leq N} \langle D^2g(y) e^{in\cdot x},e^{in\cdot x} \rangle_{E_N}. 
\end{equation*}
To show that $G^N_{\frac{c_1}{\gamma},2m}d\mu_{\frac{a_1}{\gamma}}^N = e^{-\frac{1}{\gamma}I(y)}dy$ is an invariant measure for \eqref{EQU: low freq real imag} we need to show that, for all twice differentiable $f_1,f_2:E_N\rightarrow \R$, $f=(f_1,f_2)$,
\begin{equation*}
    \int_{E_N}L_Nf(y)e^{-\frac{1}{\gamma}I(y)}dy=0.
\end{equation*}
We do this by integration by parts. We have
\begin{align*}
   \int_{E_N}L_Nf(y)e^{-\frac{1}{\gamma}I(y)}dy &= \gamma\int_{E_N} \big(\textnormal{Tr}D^2f_1(y)+ \textnormal{Tr}D^2f_2(y)\big)e^{-\frac{1}{\gamma}I(y)}dy\\
   &\hspace{10pt}-\frac{a_2}{a_1}\int_{E_N}\langle Df(y),JDI(y)\rangle_{E_N}e^{-\frac{1}{\gamma}I(y)}dy\\
   &\hspace{10pt}-\int_{E_N}\langle Df(y),DI(y)\rangle_{E_N}e^{-\frac{1}{\gamma}I(y)}dy\\
   &= (\textnormal{I})+(\textnormal{II})+(\textnormal{III}).
\end{align*}
Integrating by parts we have
\begin{align*}
    (\textnormal{II}) = -\gamma\frac{a_2}{a_1}\int_{E_N}\textnormal{Tr}(DJD)e^{-\frac{1}{\gamma}I(y)}dy = 0
\end{align*}
as $\textnormal{Tr}(DJD)$=0.

For $(\textnormal{III})$ we have,
\begin{align*}
    (\textnormal{III}) &= \gamma\int_{E_N}\langle Df(y),D(e^{-\frac{1}{\gamma}I(y)})\rangle_{E_N}dy\\
    & =  \gamma\int_{E_N}(\textnormal{Tr}D^2f_1(y)+ \textnormal{Tr}D^2f_2(y)\big)e^{-\frac{1}{\gamma}I(y)}dy\\
    &=-(\textnormal{I}). 
\end{align*}
This completes the proof.
\end{proof}

\subsection{Almost sure global well-posedness}

Before we get to the main estimate in this section, we first state some preliminary probabilistic estimates.

\begin{proposition}\label{PROP: Same laws}
Let $a_1,\gamma>0$ and $N\in \mathbb{N}$. Then,
\begin{equation*}
    \int_{C^{-\eps}} \norm{u_0}_{C^{-\eps}}^pd\mu_\frac{a_1}{\gamma}(u_0)<C<\infty
\end{equation*}
and
\begin{align*}
    \int_{C^{-\eps}}\norm{L^{(1)}_{m-1}(|\Slow u_0|^2;\sigma_N) \Slow u_0}_{C^{-\eps}}^pd\mu_{\frac{a_1}{\gamma}}(u_0) < C<\infty
\end{align*}
for some constant independent of $N$.
\end{proposition}
\begin{proof}
This result is a consequence of Proposition \ref{PROP:sconv} and the fact that the law of $\Psi(t)$ is $\mu_\frac{a_1}{\gamma}$.
\end{proof}

To prove Theorem \ref{Th: AS GWP} we use an invariant measure argument as in \cite{DPD}. The problem however is, that in order for to prove invariance of \eqref{EQU: inv meas}, we need \eqref{EQU: renormalised equ for u} to have a well-defined flow but in order to prove \eqref{EQU: renormalised equ for u} has a well defined flow we need to use the invariance of the measure \eqref{EQU: inv meas}. To enter this loop we use the truncated equation \eqref{EQU: SMOOTH renormalised equ for u_N} which we know has a globally defined flow and invariant measure \eqref{EQU: trunc meas}. We first prove the following.

\begin{proposition}\label{PROP: truncated a priori prob bound}
Suppose $\frac{a_1}{a_2} = \frac{c_1}{c_2}$. Then for $T>0$ there exists a constant $C_T>0$ such that for all $N\in \mathbb{N}$,
\begin{equation*}
    \int_{C^{-\eps}}\E[ \sup\limits_{t\in[0,T]} \norm{\Phi_N(t,u_0,\omega)}_{C^{-\eps}}]\,dP_{a_1,c_1,\gamma,2m}^N\leq C_T.
\end{equation*}
\end{proposition}
\begin{proof}
We follow an argument similar to the one in \cite{DPD}. See also \cite[Proof of Theorem 1]{DDF}. Let $u_N = \Phi_N(t,u_0,\omega)$ be the solution to \eqref{EQU: SMOOTH renormalised equ for u_N}. Then $u_N$ satisfies the mild formulation
\begin{equation*}
    u_N(t)=S(t)u_0+\int_0^tS(t-t')\mathcal{N}(u_N)dt'+ \Psi(t).
\end{equation*}
Taking the $C^{-\eps}$ norm of both sides and using Proposition \ref{PROP: Heat smooth} and the $C^{-\eps}\rightarrow C^{-\eps}$ boundedness of $S_N$ we get,
\begin{align*}
    \norm{u_N(t)}_{C^{-\eps}}\lesssim \norm{u_0}_{C^{-\eps}}+\int_0^t\norm{L^{(1)}_{m-1}(|\Slow u_N|^2;\sigma_N) \Slow u_N(t')}_{C^{-\eps}}dt'+\norm{\Psi(t)}_{C^{-\eps}}.
\end{align*}
Taking the supremum over $[0,T]$ of both sides we get,
\begin{equation*}
    \sup\limits_{t\in[0,T]}\norm{u_N(t)}_{C^{-\eps}}\lesssim \norm{u_0}_{C^{-\eps}}+\int_0^T\norm{L^{(1)}_{m-1}(|\Slow u_N|^2;\sigma_N) \Slow u_N(t')}_{C^{-\eps}}dt'+ \sup\limits_{t\in[0,T]}\norm{\Psi(t)}_{C^{-\eps}}.
\end{equation*}
Taking the expectation of both sides, integrating both sides over $C^{-\eps}$ with respect to $dP_{a_1,c_1,\gamma,2m}$ and then using Tonelli's Theorem to switch the order of integration in the second term gives,
\begin{align*}
    \int_{C^{-\eps}}\E[\sup\limits_{t\in[0,T]}\norm{u_N(t)}_{C^{-\eps}}]&dP_{a_1,c_1,\gamma,2m}^N(u_0)\lesssim \\
    &\lesssim \int_{C^{-\eps}} \norm{u_0}_{C^{-\eps}}dP_{a_1,c_1,\gamma,2m}^N(u_0) \hphantom{=}+\E\big[\sup\limits_{t\in[0,T]}\norm{\Psi(t)}{C^{-\eps}}\big]\\
    &\hphantom{=}+\int_0^T\int_{C^{-\eps}}\E\big[\norm{L^{(1)}_{m-1}(|\Slow u_N|^2;\sigma_N) \Slow u_N(t')}_{C^{-\eps}}\big]dP_{a_1,c_1,\gamma,2m}^Ndt'\\
    & = (\textnormal{I}) + (\textnormal{II}) + (\textnormal{III}) .
\end{align*}
We will estimate the three terms on the right hand of the above inequality separately and show
\begin{equation*}
    (\textnormal{I}) +(\textnormal{II}) + (\textnormal{III}) \leq C_T.
\end{equation*}
We can estimate $(\textnormal{I})$ using H\"older's inequality and Propositions \ref{PROP: Same laws} and \ref{PROP: G bound indep N},
\begin{align*}
    (\textnormal{I}) &= \int_{C^{-\eps}} \norm{u_0}_{C^{-\eps}}G^N_{\frac{c_1}{\gamma},2m}(u_0)d\mu_{\frac{a_1}{\gamma}}(u_0)\\
    &\lesssim  \bigg(\int_{C^{-\eps}}\norm{u_0}_{C^{-\eps}}^2d\mu_{\frac{a_1}{\gamma}}(u_0)\bigg)^{\frac{1}{2}}\bigg(\int_{C^{-\eps}}G^N_{\frac{c_1}{\gamma},2m}(u_0)^2d\mu_{\frac{a_1}{\gamma}}(u_0) \bigg)^{\frac{1}{2}}\\
    &\lesssim C
\end{align*}
independently of $N$. We can estimate $(\textnormal{II})$ using Proposition \ref{PROP:sconv}. For $(\textnormal{III})$ we claim that 
\begin{align*}
    (\textnormal{III}) = T\int_{C^{-\eps}} \norm{L^{(1)}_{m-1}(|\Slow u_0|^2;\sigma_N) \Slow u_0}_{C^{-\eps}}dP_{a_1,c_1,\gamma,2m}^N(u_0).
\end{align*}
To prove this claim we use an argument in \cite[Proof of Theorem 1]{DDF}. Defining the functions $\psi_m: C^{-\eps} \mapsto \C$  by $\psi_m(f) = \max(m, \norm{L^{(1)}_{m-1}(|S_N f|^2; \sigma_N)f}_{C^{-\eps}})$ we have that $\psi_m$ are Borel and bounded. Using the dominated convergence theorem, the definition of the transition semi-group, the invariance of $P_{a_1,c_1,\gamma,2m}^N$ and then the dominated convergence theorem again we get,  
\begin{align*}
    (\textnormal{II}) & = \int_0^T\int_{C^{-\eps}}\E\big[\norm{L^{(1)}_{m-1}(|\Slow u_N|^2;\sigma_N) \Slow u_N(t')}_{C^{-\eps}}\big]dP_{a_1,c_1,\gamma,2m}^Ndt'\\
    & = \int_0^T \lim_{m\rightarrow\infty}\int_{C^{-\eps}} P_{t'}^N\psi_mdP_{a_1,c_1,\gamma,2m}^N dt'\\
    & = \int_0^T \lim_{m\rightarrow\infty}\int_{C^{-\eps}} \psi_m  dP_{a_1,c_1,\gamma,2m}^N dt'\\
    & = T\int_{C^{-\eps}} \norm{L^{(1)}_{m-1}(|\Slow u_0|^2;\sigma_N) \Slow u_0}_{C^{-\eps}}dP_{a_1,c_1,\gamma,2m}^N(u_0).
\end{align*}
From H\"older's inequality and  Propositions \ref{PROP: Same laws} and \ref{PROP: G bound indep N} we then have,
\begin{align*}
    (\textnormal{II}) &\lesssim T\int_{C^{-\eps}} \norm{L^{(1)}_{m-1}(|\Slow u_0|^2;\sigma_N) \Slow u_0}_{C^{-\eps}}dP_{a_1,c_1,\gamma,2m}^N(u_0)\\
    & \lesssim C_T\bigg(\int_{C^{-\eps}}\norm{L^{(1)}_{m-1}(|\Slow u_0|^2;\sigma_N) \Slow u_0}_{C^{-\eps}}^2d\mu_{\frac{a_1}{\gamma}}(u_0)\bigg)^{\frac{1}{2}}\bigg(\int_{C^{-\eps}}G^N_{\frac{c_1}{\gamma},2m}(u_0)^2d\mu_{\frac{a_1}{\gamma}}(u_0) \bigg)^{\frac{1}{2}}\\
    &\lesssim C_T
\end{align*}
giving the bound for $(\textnormal{III})$.

\end{proof}

We now upgrade the previous proposition to the following.

\begin{proposition}\label{PROP: a priori prob bound}
Suppose $\frac{a_1}{a_2} = \frac{c_1}{c_2}$. For any $T>0$ there exists a constant $C_T>0$ such that
\begin{equation*}
    \int_{C^{-\eps}}\E\big[ \sup\limits_{t\in[0,T\wedge T^*(u_0,\omega)} \norm{\Phi(t,u_0,\omega)}_{C^{-\eps}}\big]\,dP_{a_1,c_1,\gamma,2m}(u_0)\leq C_T.
\end{equation*}
Here $T^*(u_0, \omega)>0$ a.s. is the maximal time of existence of $u$ according to Proposition \ref{LWP}. 
\end{proposition}
\begin{proof}
From Proposition \ref{PROP: G bound indep N} we know that $G^N_{c,2m}(u_0)\rightarrow G_{c,2m}(u_0)$  $P_{a_1,c_1,\gamma,2m}$-almost everywhere, up to a sub sequence which we still denote by $\{G^N_{c,2m}\}$ by abuse of notation. Further, by Proposition \ref{PROP: vN approximates v}, $u_N\rightarrow u$ in $C([0, T \wedge T^*(u_0,\omega)], C^{-\eps})$ where $u$ solves \eqref{EQU: renormalised equ for u} with initial data $u_0$ and $u_N=\Phi_N(t,u_0,\omega)$ solves \eqref{EQU: SMOOTH trunc v SCGL}. 

Hence by Fatou's lemma we have,
\begin{align*}
     \int_{C^{-\eps}}\E\big[ \sup\limits_{t\in[0,T\wedge T^*(u_0)} &\norm{\Phi(t,u_0,\omega)}_{C^{-\eps}}\big]\,dP_{a_1,c_1,\gamma,2m}(u_0)\\ 
     & = \int_{C^{-\eps}}\E\big[ \sup\limits_{t\in[0,T\wedge T^*(u_0)} \norm{u(t)}_{C^{-\eps}}\big]G_{\frac{c_1}{\gamma}, 2m}(u_0)\,d\mu_{\frac{a_1}{\gamma}}(u_0)\\
     & \lesssim \liminf\limits_{N\rightarrow\infty} \int_{C^{-\eps}}\E\big[ \sup\limits_{t\in[0,T\wedge T^*(u_0)} \norm{u_N(t)}_{C^{-\eps}}\big]G_{\frac{c_1}{\gamma}, 2m}^N(u_0)\,d\mu_{\frac{a_1}{\gamma}}(u_0)\\
     &< C_T.
\end{align*}
\end{proof}

We can now complete the proof of Theorem \ref{Th: AS GWP}.

\begin{proof}[Proof of Theorem \ref{Th: AS GWP}]
Consider a sequence of times $T_k\rightarrow \infty$ as $k\rightarrow\infty$. Then, by the above proposition, for each $k$ there exists a set $U_k$ of $P_{\frac{a_1}{\gamma},\frac{c_1}{\gamma},2m}$-measure one such that for each $u_0\in U_k$ 
\begin{equation*}
    \sup\limits_{t\in[0,T\wedge T^*(u_0)]}\norm{u(t)}_{C^{-\eps}}<\infty\quad \textnormal{almost surely.}
\end{equation*}
From the ansatz $u=v+\Psi$ and Proposition \ref{PROP:sconv} we then have
\begin{equation*}
    \sup\limits_{t\in[0,T\wedge T^*(u_0)]}\norm{v(t)}_{C^{-\eps}}<\infty\quad \textnormal{almost surely.}
\end{equation*}
where $v(t)$ is the solution to \eqref{EQU: for v SCGL}. Proposition \ref{solntoSCGL11} then implies that $T^*(u_0)\geq T_k$ almost surely. Taking $U=\bigcap\limits_{k\geq 0}U_k$ we have our desired measure one set. 
\end{proof}

This almost sure global existence result allows us to define the solution map $\Phi:\R_+\times U\times \Omega \rightarrow C^{-\eps}$ associated to \eqref{SCGL11} with initial data in $U$. Further we can define the transition semi-group $P_t$ associated to \eqref{SCGL11} by
\begin{equation*}
    P_t\psi(u_0) = \E\big[\psi(\Phi(t,u_0,\omega)  \big]
\end{equation*}

Finally we show that the measure $P_{a_1,c_1,\gamma,2m}$ is an invariant measure for \eqref{SCGL11}, proving Theorem \ref{THM: invariance}.

\begin{theorem}
Suppose $\frac{a_1}{a_2} = \frac{c_1}{c_2}$. Then the measure $P_{a_1,c_1,\gamma,2m}$ is an invariant measure for \eqref{SCGL11}. That is,
\begin{equation*}
    \int_{C^{-\eps}}P_t\psi\,dP_{a_1,c_1,\gamma,2m} = \int_{C^{-\eps}}\psi\,dP_{a_1,c_1,\gamma,2m}
\end{equation*}
for all $t>0$ where $P_t$ is the transition semi-group associated to \eqref{SCGL11}.
\end{theorem}
\begin{proof}
We follow an approximation argument as in \cite{DDF}. Let $\psi \in C_b(C^{-\eps})$. From the invariance of $dP_{a_1,c_1,\gamma,2m}^N$,
\begin{equation}\label{EQU: trunc invariance}
    \int_{C^{-\eps}}P_t\psi\,dP_{a_1,c_1,\gamma,2m}^N = \int_{C^{-\eps}}\psi\,dP_{a_1,c_1,\gamma,2m}^N.
\end{equation}
For the left hand side of \eqref{EQU: trunc invariance} we have,
\begin{align}\label{EQU: left hand side Trans semi group}
    \bigg|\int_{C^{-\eps}}&P_t\psi\,dP_{a_1,c_1,\gamma,2m}-\int_{C^{-\eps}}P_t^N\psi\,dP_{a_1,c_1,\gamma,2m}^N\bigg|\\
    &\leq \bigg|\int_{C^{-\eps}}\E\big[\psi(\Phi(t,u_0,\omega))\big]dP_{a_1,c_1,\gamma,2m}-\int_{C^{-\eps}}\E\big[\psi(\Phi(t,u_0,\omega))\big]dP_{a_1,c_1,\gamma,2m}^N\bigg|\nonumber\\
    &\hspace{10pt}+\bigg|\int_{C^{-\eps}}\big(\E\big[\psi(\Phi(t,u_0,\omega))\big] - \E\big[\psi(\Phi_N(t,u_0,\omega))\big]\big) dP_{a_1,c_1,\gamma,2m}\bigg|\nonumber\\
    &\leq \norm{\psi}_{L^\infty}\big| Z_{a_1,c_1,\gamma,2m}-Z_{a_1,c_1,\gamma,2m}^N\big|\nonumber\\
    &\hspace{10pt}+\int_{C^{-\eps}}\big|\E\big[\psi(\Phi(t,u_0,\omega))\big] - \E\big[\psi(\Phi_N(t,u_0,\omega))\big]\big| dP_{a_1,c_1,\gamma,2m}\nonumber.
\end{align}
The first term on the right hand side of \eqref{EQU: left hand side Trans semi group} goes to zero from Proposition \ref{PROP: G bound indep N}. The second term on the right hand side of \eqref{EQU: left hand side Trans semi group} goes to zero from a combination of the fact that $\Phi_N(t,u_0,\omega)\rightarrow \Phi_N(t,u_0,\omega)$ in $C([0,T];C^{-\eps})$ for almost every $u_0\in \supp P_{a_1,c_1,\gamma,2m}$, the continuity of $\psi$ and the dominated convergence theorem.

For the right hand side of \eqref{EQU: trunc invariance},
\begin{align*}
\bigg|\int_{C^{-\eps}}\psi\,dP_{a_1,c_1,\gamma,2m}&-\int_{C^{-\eps}}\psi\,dP_{a_1,c_1,\gamma,2m}^N\bigg|\nonumber\\    &\lesssim ({Z_{a_1,c_1,\gamma,2m}^N})^{-1}\bigg|\int_{C^{-\eps}}\psi\, G_{\frac{c_1}{\gamma},2m}d\mu_\frac{a_1}{\gamma} -\int_{C^{-\eps}}\psi\, G_{\frac{c_1}{\gamma},2m}^Nd\mu_\frac{a_1}{\gamma}\bigg|\nonumber\\
&\hspace{10pt}+\big(({Z_{a_1,c_1,\gamma,2m}^N})^{-1}-({Z_{a_1,c_1,\gamma,2m}^N})^{-1}\big)\int_{C^{-\eps}}\psi\,G_{\frac{c_1}{\gamma},2m}d\mu_{\frac{a_1}{\gamma}}\nonumber\\
&\leq \delta^{-1}\norm{\psi}_{L^\infty}\big|Z_{a_1,c_1,\gamma,2m}-Z_{a_1,c_1,\gamma,2m}^N\big|\nonumber\\
&\hspace{10pt}+\delta^{-2}\big|Z_{a_1,c_1,\gamma,2m}-Z_{a_1,c_1,\gamma,2m}^N\big|\norm{\psi}_{L^\infty}\delta\nonumber\\
&\leq \delta^{-1}\norm{\psi}_{L^\infty}\big|Z_{a_1,c_1,\gamma,2m}-Z_{a_1,c_1,\gamma,2m}^N\big|
\end{align*}
where in the second inequality $\delta$ is the constant in \eqref{EQU: Z bounded below}. The right hand side of this inequality goes to zero from Proposition \ref{PROP: G bound indep N}.
\end{proof}

\begin{ackno}\rm
The author would like to thank his supervisor Tadahiro Oh for suggesting this problem and the support in completing it. The author would also like to thank Leonardo Tolomeo for many helpful discussions related to this work and Justin Forlano for some helpful suggestions regarding the introduction of this paper.
The author was supported by The Maxwell Institute Graduate School in Analysis and its
Applications, a Centre for Doctoral Training funded by the UK Engineering and Physical
Sciences Research Council (grant EP/L016508/01), the Scottish Funding Council, Heriot-Watt
University and the University of Edinburgh.
\end{ackno}

\end{document}